\documentclass[12pt]{amsart}
\usepackage[margin=1in]{geometry}
\usepackage{amsmath,amsfonts,amsthm,amssymb,bbm}
\usepackage{graphicx,color,dsfont}
\usepackage{enumitem}
\usepackage{fourier}

\begin{document}

\newtheorem{theorem}{Theorem}
\newtheorem{lemma}{Lemma}
\newtheorem{proposition}{Proposition}
\newtheorem{rmk}{Remark}
\newtheorem{example}{Example}
\newtheorem{exercise}{Exercise}
\newtheorem{definition}{Definition}
\newtheorem{corollary}{Corollary}
\newtheorem{notation}{Notation}
\newtheorem{claim}{Claim}

\newtheorem{dif}{Definition}

 \newtheorem{thm}{Theorem}[section]
 \newtheorem{cor}[thm]{Corollary}
 \newtheorem{lem}[thm]{Lemma}
 \newtheorem{prop}[thm]{Proposition}
 \theoremstyle{definition}
 \newtheorem{defn}[thm]{Definition}
 \theoremstyle{remark}
 \newtheorem{rem}[thm]{Remark}
 \newtheorem*{ex}{Example}
 \numberwithin{equation}{section}

\newcommand{\vertiii}[1]{{\left\vert\kern-0.25ex\left\vert\kern-0.25ex\left\vert #1
    \right\vert\kern-0.25ex\right\vert\kern-0.25ex\right\vert}}

\newcommand{\R}{{\mathbb R}}
\newcommand{\C}{{\mathbb C}}
\newcommand{\U}{{\mathcal U}}
\newcommand{\norm}[1]{\left\|#1\right\|}
\renewcommand{\(}{\left(}
\renewcommand{\)}{\right)}
\renewcommand{\[}{\left[}
\renewcommand{\]}{\right]}
\newcommand{\f}[2]{\frac{#1}{#2}}
\newcommand{\im}{i}
\newcommand{\cl}{{\mathcal L}}
\newcommand{\ck}{{\mathcal K}}

\newcommand{\al}{\alpha}
\newcommand{\vro}{\varrho}
\newcommand{\be}{\beta}
\newcommand{\wh}[1]{\widehat{#1}}
\newcommand{\ga}{\gamma}
\newcommand{\Ga}{\Gamma}
\newcommand{\de}{\delta}
\newcommand{\ben}{\beta_n}
\newcommand{\De}{\Delta}
\newcommand{\ve}{\varepsilon}
\newcommand{\ze}{\zeta}
\newcommand{\Th}{\Theta}
\newcommand{\ka}{\kappa}
\newcommand{\la}{\lambda}
\newcommand{\laj}{\lambda_j}
\newcommand{\lak}{\lambda_k}
\newcommand{\La}{\Lambda}
\newcommand{\si}{\sigma}
\newcommand{\Si}{\Sigma}
\newcommand{\vp}{\varphi}
\newcommand{\om}{\omega}
\newcommand{\Om}{\Omega}
\newcommand{\ra}{\rightarrow}

\newcommand{\ro}{{\mathbb  R}}
\newcommand{\rn}{{\mathbb  R}^n}
\newcommand{\rd}{{\mathbb  R}^d}
\newcommand{\rmm}{{\mathbb  R}^m}
\newcommand{\rone}{\mathbb  R}
\newcommand{\rtwo}{\mathbb  R^2}
\newcommand{\rthree}{\mathbb  R^3}
\newcommand{\rfour}{\mathbb  R^4}
\newcommand{\ronen}{{\mathbb  R}^{n+1}}
\newcommand{\ku}{\mathbb  u}
\newcommand{\kw}{\mathbb  w}
\newcommand{\kf}{\mathbb  f}
\newcommand{\kz}{\mathbb  z}

\newcommand{\N}{\mathbb  N}

\newcommand{\tn}{\mathbb  T^n}
\newcommand{\tone}{\mathbb  T^1}
\newcommand{\ttwo}{\mathbb  T^2}
\newcommand{\tthree}{\mathbb  T^3}
\newcommand{\tfour}{\mathbb  T^4}

\newcommand{\zn}{\mathbb  Z^n}
\newcommand{\zp}{\mathbb  Z^+}
\newcommand{\zone}{\mathbb  Z^1}
\newcommand{\zz}{\mathbb  Z}
\newcommand{\ztwo}{\mathbb  Z^2}
\newcommand{\zthree}{\mathbb  Z^3}
\newcommand{\zfour}{\mathbb  Z^4}

\newcommand{\hn}{\mathbb  H^n}
\newcommand{\hone}{\mathbb  H^1}
\newcommand{\htwo}{\mathbb  H^2}
\newcommand{\hthree}{\mathbb  H^3}
\newcommand{\hfour}{\mathbb  H^4}

\newcommand{\cone}{\mathbb  C^1}
\newcommand{\ctwo}{\mathbb  C^2}
\newcommand{\cthree}{\mathbb  C^3}
\newcommand{\cfour}{\mathbb  C^4}
\newcommand{\dpr}[2]{\langle #1,#2 \rangle}

\newcommand{\sn}{\mathbb  S^{n-1}}
\newcommand{\sone}{\mathbb  S^1}
\newcommand{\stwo}{\mathbb  S^2}
\newcommand{\sthree}{\mathbb  S^3}
\newcommand{\sfour}{\mathbb  S^4}

\newcommand{\lp}{L^{p}}
\newcommand{\lppr}{L^{p'}}
\newcommand{\lqq}{L^{q}}
\newcommand{\lr}{L^{r}}
\newcommand{\echi}{(1-\chi(x/M))}
\newcommand{\chip}{\chi'(x/M)}

\newcommand{\wlp}{L^{p,\infty}}
\newcommand{\wlq}{L^{q,\infty}}
\newcommand{\wlr}{L^{r,\infty}}
\newcommand{\wlo}{L^{1,\infty}}

\newcommand{\lprn}{L^{p}(\rn)}
\newcommand{\lptn}{L^{p}(\tn)}
\newcommand{\lpzn}{L^{p}(\zn)}
\newcommand{\lpcn}{L^{p}(\cn)}
\newcommand{\lphn}{L^{p}(\cn)}

\newcommand{\lprone}{L^{p}(\rone)}
\newcommand{\lptone}{L^{p}(\tone)}
\newcommand{\lpzone}{L^{p}(\zone)}
\newcommand{\lpcone}{L^{p}(\cone)}
\newcommand{\lphone}{L^{p}(\hone)}

\newcommand{\lqrn}{L^{q}(\rn)}
\newcommand{\lqtn}{L^{q}(\tn)}
\newcommand{\lqzn}{L^{q}(\zn)}
\newcommand{\lqcn}{L^{q}(\cn)}
\newcommand{\lqhn}{L^{q}(\hn)}

\newcommand{\lo}{L^{1}}
\newcommand{\lt}{L^{2}}
\newcommand{\li}{L^{\infty}}
\newcommand{\beqn}{\begin{eqnarray*}}
\newcommand{\eeqn}{\end{eqnarray*}}
\newcommand{\pplus}{P_{Ker[\cl_+]^\perp}}

\newcommand{\co}{C^{1}}
\newcommand{\ci}{\mathcal I}
\newcommand{\coi}{C_0^{\infty}}

\newcommand{\ca}{\mathcal A}
\newcommand{\cs}{\mathcal S}
\newcommand{\cm}{\mathcal M}
\newcommand{\cf}{\mathcal F}
\newcommand{\cb}{\mathcal B}
\newcommand{\ce}{\mathcal E}
\newcommand{\cd}{\mathcal D}
\newcommand{\cn}{\mathbb N}
\newcommand{\cz}{\mathcal Z}
\newcommand{\crr}{\mathbb  R}
\newcommand{\cc}{\mathcal C}
\newcommand{\ch}{\mathcal H}
\newcommand{\cq}{\mathcal Q}
\newcommand{\cp}{\mathcal P}
\newcommand{\cx}{\mathcal X}
\newcommand{\eps}{\epsilon}

\newcommand{\pv}{\textup{p.v.}\,}
\newcommand{\loc}{\textup{loc}}
\newcommand{\intl}{\int\limits}
\newcommand{\iintl}{\iint\limits}
\newcommand{\dint}{\displaystyle\int}
\newcommand{\diint}{\displaystyle\iint}
\newcommand{\dintl}{\displaystyle\intl}
\newcommand{\diintl}{\displaystyle\iintl}
\newcommand{\liml}{\lim\limits}
\newcommand{\suml}{\sum\limits}
\newcommand{\ltwo}{L^{2}}
\newcommand{\supl}{\sup\limits}
\newcommand{\df}{\displaystyle\frac}
\newcommand{\p}{\partial}
\newcommand{\Ar}{\textup{Arg}}
\newcommand{\abssigk}{\widehat{|\si_k|}}
\newcommand{\ed}{(1-\p_x^2)^{-1}}
\newcommand{\tT}{\tilde{T}}
\newcommand{\tV}{\tilde{V}}
\newcommand{\wt}{\widetilde}
\newcommand{\Qvi}{Q_{\nu,i}}
\newcommand{\sjv}{a_{j,\nu}}
\newcommand{\sj}{a_j}
\newcommand{\pvs}{P_\nu^s}
\newcommand{\pva}{P_1^s}
\newcommand{\cjk}{c_{j,k}^{m,s}}
\newcommand{\Bjsnu}{B_{j-s,\nu}}
\newcommand{\Bjs}{B_{j-s}}
\newcommand{\Ly}{\cl_+i^y}
\newcommand{\dd}[1]{\f{\partial}{\partial #1}}
\newcommand{\czz}{Calder\'on-Zygmund}
\newcommand{\chh}{\mathcal H}

\newcommand{\lbl}{\label}
\newcommand{\beq}{\begin{equation}}
\newcommand{\eeq}{\end{equation}}
\newcommand{\beqna}{\begin{eqnarray*}}
\newcommand{\eeqna}{\end{eqnarray*}}
\newcommand{\bp}{\begin{proof}}
\newcommand{\ep}{\end{proof}}
\newcommand{\bprop}{\begin{proposition}}
\newcommand{\eprop}{\end{proposition}}
\newcommand{\bt}{\begin{theorem}}
\newcommand{\et}{\end{theorem}}
\newcommand{\bex}{\begin{Example}}
\newcommand{\eex}{\end{Example}}
\newcommand{\bc}{\begin{corollary}}
\newcommand{\ec}{\end{corollary}}
\newcommand{\bcl}{\begin{claim}}
\newcommand{\ecl}{\end{claim}}
\newcommand{\bl}{\begin{lemma}}
\newcommand{\el}{\end{lemma}}
\newcommand{\dea}{(-\De)^\be}
\newcommand{\naa}{|\nabla|^\be}
\newcommand{\cj}{{\mathcal J}}
\newcommand{\ubb}{{\mathbb  u}}

\title[Solitary waves for the  power degenerate NLS]{Solitary waves for the power degenerate NLS - existence and stability}

\author{Vishnu Iyer}
\address{Department of Mathematics,
	University of Alabama - Birmingham, 
	University Hall, Room 4050, 
	1402 10th Avenue South
	Birmingham AL 35294-1241
	 }\email{vpiyer@uab.edu}
 
\author[Atanas G. Stefanov]{\sc Atanas G. Stefanov}
\address{ Department of Mathematics,
	University of Alabama - Birmingham, 
	University Hall, Room 4005, 
	1402 10th Avenue South
	Birmingham AL 35294-1241
	 }
\email{stefanov@uab.edu}



\date{\today}
 
\begin{abstract}
 We consider a semilinear Schr\"odinger equation, driven by the power degenerate second order differential operator $\nabla\cdot (|x|^{2a} \nabla), a\in (0,1)$. We construct the solitary waves, in the sharp range of parameters, as minimizers of the Caffarelli-Kohn-Nirenberg's inequality. Depending on the parameter $a$ and the nonlinearity, we establish a number of properties, such as positivity, smoothness (away from the origin) and almost exponential decay. Then, and as a consequence of our variational constrcution, we completely characterize the spectral stability of the said solitons. 
 
 We pose some natural  conjectures,  which are still open - such as the radiality of the ground states, the non-degeneracy and most importantly uniqueness. 
\end{abstract}

\thanks{ Iyer acknowledges partial support  through a graduate research fellowship from NSF-DMS \# 2204788. Stefanov is partially supported by   NSF-DMS \# 2204788.}

\maketitle

\section{Introduction}

In the present paper, we consider the  following degenrate  semi-linear  Schrodinger
equation,
\begin{equation}
	\label{14}
	\begin{cases}
	&	iu_{t}+\nabla\cdot(|x|^{2a}\nabla u)+|u|^{p-1}u=  0, x\in \rd, t\in \rone \\
	&	u(0,x)=  u_{0}(x)
	\end{cases}
\end{equation}
where $u:\mathbb {R_{+}}\times\mathbb {R}^{d}\to\mathbb {C}$. Many physical processes, such as in plasma physics and more magnetism models for configurations, away from equilibrium,  are driven by Schr\"odinger equations,  $i \p_t+\nabla\cdot(\si(x)\nabla u)$ with degenerate $\si(x)$,  \cite{DPL}. For more mathematically oriented treatments, see \cite{DL, Ich}. 

In \cite{PTW}, the authors developed a high regularity local well-posedness theory for \eqref{14}, with pretty general divergence form operator, $\nabla\cdot(\si(x)\nabla u)$, but only on $\rtwo$. It seems that general well-posedness theory is lacking at present, and in particular such important issues such as global solutions vs. blow-up alternative, criteria for finite time blow up, etc. remain open. At the same time, and at least formally (i.e. for smooth enough solutions),  
the model \eqref{14}  obeys conservation laws such as mass
and  energy, given by 
\begin{eqnarray}
	\label{182}
	M(t) &:=& \int|u(t,x)|^{2}\ dx=M(0) \\
	\label{24}
	E(t) &:= & \frac{1}{2}\int|x|^{2a}|\nabla u(t,x)|^{2}\ dx-\frac{1}{p+1}\int|u|^{p+1}\ dx=E(0).
\end{eqnarray}
\subsection{Standing wave solutions}
Our main interest in this work will be in the existence and properties of  solitary wave solutions of the form $u(t,x)=e^{i\omega t}\phi(x)$,
where $\phi(x)$ is a real valued function. Observe that any such
classical solutions $\phi$ should satisfy the profile equation, 
\begin{equation}
	-\nabla\cdot(|x|^{2a}\nabla\phi)+\omega\phi-|\phi|^{p-1}\phi=0.
	\label{120}
\end{equation}
We discuss the existence of these waves (under essentially sharp necessary and sufficient conditions on the parameters), along with appropriate pointwise decay properties, see Theorem \ref{theo:10} below. We would like now to focus our attention to the stability of the solutions 
$e^{i\omega t}\phi$ in the framework of the time-dependent problem \eqref{14}. 
We should mention the work \cite{ZZ}, where solitary waves for a more general models were considered, namely, with an inhomogeneous non-linearity $|x|^c |u|^p u$. The obtained results therein concern instability by blow-up\footnote{see the precise statement in Definition \ref{defi:18}} for such waves, which do occur for specific values of the parameters (see for example our Theorem \ref{theo:30} below). Our goal in this paper  is to provide a consistent and exhaustive classification of all the waves occuring in \eqref{120}. 
\subsection{Linearized problem} 
To this end, we look at the linearized problem first. We set $u=e^{i \om t} (\phi+ v)$ and plug it in \eqref{14}. After splitting in real and imaginary parts, $v=v_1+i v_2$ and ignoring $O(v^2)$ terms, we arrive at the following linearized problem 
\begin{equation}
	\label{l:10} 
	\vec{v}_t= \cj \cl \vec{v}, 
\end{equation}
where 
\begin{eqnarray*}
\cj &=& \left(\begin{array}{cc} 
	0 & -1 \\ 
	 1 & 0 
\end{array}\right), \ \ \cl=\left(\begin{array}{cc} 
\cl_+ & 0 \\ 
0 & \cl_-
\end{array}\right), \\
	\cl_+ &=& -\nabla\cdot(|x|^{2a}\nabla  )+\om- p \phi^{p-1} \\
	\cl_-  &=&  -\nabla\cdot(|x|^{2a}\nabla )+\om-  \phi^{p-1} 
\end{eqnarray*} 
As the linear evolution of this linearized problem is driven by the semi-group generated by $\cj\cl$, we  say that the solitary wave solution $e^{i \om t} \phi$ of  \eqref{14} is {\it spectrally stable}, if the eigenvalue problem 
\begin{equation}
	\label{l:20} 
\cj \cl \vec{v}=\la \vec{v}
\end{equation}
does not have non-trivial solutions\footnote{The definition of $D(\cl)$ to be made precise below} $(\la, \vec{v}): \Re \la>0, \vec{v}\in D(\cl)$.

\subsection{Existence results}
The problem for the existence of the waves, i.e. \eqref{120} and even more general models,  has been studied quite extensively in the last thirty or so years. In this context, we would like to mention the works \cite{Cha, CHCH, DLY, GGW}. Most of these works are purely elliptic PDE contributions, as they employ variational methods, based on the powerful machinery of the mountain pass principle. Such constructions are, generally speaking, not well-suited for stability considerations. The main reason is that one needs certain specific {\it a posteriori} information the linearized operators arising in these considerations, which is unavailable through the mountain pass formalism. 

 In our approach, we use different, constrained optimization method, namely the Weinstein method. The main idea is to constrain the potential energy and maximize the kinetic energy\footnote{One can take,  the completely equivalent reverse route, which is to constrain the kinetic energy and maximize the potential energy. We choose the former method for a slight technical advantage and convenience}.  In addition, we show that the waves are smooth away from the origin and also, we develop sharp decay estimates as $|x|\to +\infty$.  Finally, under the conjectured radiality  on the wave,  we provide a precise asymptotics at $x=0$. It  shows that as long as $a\neq 0$, the function cannot be $C^2(\rd)$. 
 
 More precisely, we have the following existence theorem. 
\begin{theorem}
	\label{theo:10} 
	Let $d\geq 1$ and $a, p>1$ satisfy  the relation
	\begin{equation}
		\label{925} 
		0\leq a<1,  \f{1}{2}  - \f{1}{p+1} < \f{1-a}{d}. 
	\end{equation}
	Then, the equation \eqref{120} has a distributional positive solution $\phi:  \phi\in C^\infty(\rd\setminus\{0\})$. 
	 Moreover, $\phi\in H^{1,a}(\rd)$ (see the definition \eqref{h1a} below),  and there is  the following  pointwise exponential bound 
	\begin{equation}
		\label{18} 
		0<\phi(x) \leq C e^{-\de |x|^{1-a}},
	\end{equation}
for some $\de>0, C>0$. 
In case the solution $\phi$ is radial, then it is also continuous at zero, and it satisfies for $0<\rho<<1$, 
\begin{equation}
	\label{26} 
	\begin{cases}
		\phi'(\rho) = - \f{\phi^p(0)-\om \phi(0)}{d} \rho^{1-2a}+o(\rho^{1-2a}) \\
		\phi''(\rho) =  \f{2a-1}{d} (\phi^p(0)-\om \phi(0)) \rho^{-2a} + o(\rho^{-2a}).
	\end{cases}, \ \textup{where} \ \ \phi^p(0)-\om \phi(0)>0.
\end{equation} 
In particular, if $a\in (0, 1/2]$, $\phi\in C^1(0, \infty)$, while for $a\in (1/2,1)$, $\phi'(\rho)$ blows up as $\rho\to 0+$. At the same time, $\phi''$ always blows up at zero, $a\in (0,1)$. 
\end{theorem}
{\bf Remark:} We conjecture that all solutions $\phi$ of \eqref{120} are indeed radial, in which case we have the asymptotics \eqref{26}. Unfortunately, we currently do not have  a proof of this, which is the reason we only provide a conditional  statement as \eqref{26}. 

Our next result are the appropriate Pohozaev's identities for the elliptic problem \eqref{120}. 
\begin{proposition}(Pohozaev's identities)
	\label{prop:Poh}
	
	Assume  $a\in [0,1), p>1$ and let $\phi\in H^{1,a}(\rd)\cap L^{p+1}$ be a  distributional solution to \eqref{120}. Then, 
	\begin{eqnarray}
		\label{500} 
		\int|x|^{2a}|\nabla\phi|^{2}\ dx &=& \frac{d(p-1)}{2(p+1)(1-a)} \int \phi^{p+1} dx \\
		\label{510} 
		\om \int \phi^2(x) dx &=&  \left(1-\frac{d(p-1)}{2(p+1)(1-a)}\right) \int \phi^{p+1} dx. 
	\end{eqnarray}
\end{proposition} 
We present a proof of this in the appendix.
An immediate corollary of Proposition \ref{prop:Poh},  is the following necessary condition for the existence of the waves. 
\begin{corollary}
	\label{cor:10} 
	Let $\om>0, p>1, a\geq 0$. Then, localized solutions  $\phi>0$ of \eqref{120} exist, only if $a<1$ and there exists $\theta\in (0,1)$, so that 
	\begin{equation}
		\label{523}
			\f{1}{p+1} = \f{1}{2}  - \f{\theta}{d} (1-a).
	\end{equation}
\end{corollary}
Indeed, the statement $a<1$ follows from the positivity of the coefficient in \eqref{500}, while by \eqref{510}, $\left(1-\frac{d(p-1)}{2(p+1)(1-a)}\right)>0$, we obtain 
\begin{equation}
	\label{527} 
	\f{1-a}{d}>\f{1}{2}-\f{1}{p+1},
\end{equation}
which is equivalent to  \eqref{523} for some $\theta\in (0,1)$. 

Based on Corollary \ref{cor:10} and  modulo the assumption\footnote{which is also likely to be necessary} $\om>0$, the existence results in Theorem \ref{theo:10} are optimal as stated. 

\subsection{Stability results}
Next, we present the spectral stability results for the waves constructed in Theorem \ref{theo:10}. 
\begin{theorem}
	\label{theo:20} 
Let $d, a, \theta, p$ and the function $\phi$ are as in Theorem \ref{theo:10}. Then, the wave $e^{i \om t} \phi$, is spectrally stable if and only if $1<p\leq 1+\f{4(1-a)}{d}$. 
\end{theorem}
Finally, we discuss the strong instability of the spectrally unstable 
waves. The proper definition is as follows. 
\begin{defn}
	\label{defi:18}
	We say that a standing wave $e^{i\omega t}\phi_{\omega}(x)$ that
	solves \eqref{14},  is strongly unstable in the Banach space $X$,  if for
	any $\epsilon>0$,  there exist $u_{0}\in X$ such that $\|u_{0}-\phi_{\omega}\|_{X}<\epsilon$
	and the solution $u(t)$ of \ref{14} with initial data $u(0)=u_{0}$
	blows-up in finite time. 
\end{defn}

Specifically, we have 
\begin{theorem}
	\label{theo:30} 
	Let $d\geq 1$, $p>1+\frac{4(1-a)}{d}$, but $p$ satisfies \eqref{527}, so that the waves $\phi$ exist.  Then, the waves $e^{i \om t} \phi$ are unstable by blow-up in the space $H^{1,a}(\rd)$. 
\end{theorem}

\section{Preliminaries}
We start with some notations and function spaces. 
\subsection{Function spaces} 
We will use the standard definition of $L^p$ spaces,
$$
\|f\|_{L^p}=\left(\int_{\rd} |f(x)|^p dx\right), 1\leq p<\infty
$$
with the obvious modification for $p=\infty$. Recall the classical Riesz-Kolmogorov criteria for compactness in $L^q(\rd)$, which states that a set ${\mathcal F}\subset L^q(\rd), 1<q<\infty$ is compact if and only if 
\begin{enumerate}
	\item $\sup_{f\in {\mathcal F}}\|f\|_{L^q}<\infty$ 
	\item For every $\eps>0$, there exists $R$, so that 
	$$
\sup_{f\in {\mathcal F} }	\int_{|x|>R} |f(x)|^q dx <\eps. 
	$$
	\item For every $\eps>0$, there exists $\de>0$, so that for all $|y|<\de$, 
	$$
	\sup_{f\in {\mathcal F} }	\int_{\rd} |f(x+y)-f(x)|^q  dx <\eps. 
		$$
\end{enumerate}
Next, the Fourier transform and its inverse are defined via 
$$
\hat{f}(\xi)=\int_{\rd} f(x) e^{-2\pi i \dpr{x}{\xi}} dx, \ \ f(x) = \int_{\rd} \hat{f}(\xi)  e^{2\pi i \dpr{x}{\xi}} d\xi, 
$$
In this setting, we see that for Schwartz functions $f$, we can realize $\widehat{-\De f}(\xi)=4\pi^2 |\xi|^2 \hat{f}(\xi)$. More generally, for any $s>0$,  the fractional Laplacian of order $s$, can be introduced via $\widehat{(-\De)^s f}(\xi)=(4\pi^2)^s  |\xi|^{2s}  \hat{f}(\xi)$.  

Next, for any $\si>0$, we can define the operator $(\si-\De)^{-1}$, defined via the Fourier transform 
$
\widehat{(\si-\De)^{-1} f}(\xi)=(\si+4\pi^2 |\xi|^2)^{-1} \hat{f}(\xi).
$
It is well-known that this operator is given by a convolution with the  Green's function $\si^{\f{d}{2}-1} Q(\sqrt{\si} x)$, which has the asymptotics 
\begin{equation}
	\label{q:10} 
	Q(x)\sim \begin{cases}
		Ce^{-|x|} & |x|>1, d=1.\\
		C|x|^{2-d} & |x|<1,\ d\geq3\\
		\ln(\frac{1}{|x|}) & |x|<1,\ d=2. 
	\end{cases}.
\end{equation}
The Sobolev spaces $W^{s,p}, 1<p<\infty, s>0$ can be taken as the completion of the set of Schwartz functions in the norm 
$$
\|f\|_{W^{s,p}}=\|(-\De)^{\f{s}{2}} f\|_{L^p}+ \|f\|_{L^p}.
$$
It is convenient to introduce the notation $D=\sqrt{-\De}$ as the symbold of order $1$. Clearly, by the definitions above, 
$\|D f\|_{L^p}\sim \|\nabla f\|_{L^p}, 1<p<\infty$.

We also need a standard product estimate on Sobolev spaces. Specifically, let $1<p,q_1,r_1, q_2, r_2<\infty$, $\f{1}{p}=\f{1}{q_1}+\f{1}{r_1}=\f{1}{q_2}+\f{1}{r_2}$ and $0<s<1$, then there exists $C$, so that  
\begin{equation}
	\label{KP} 
	\|D^s( f g)\|_{L^p}\leq C(\|D^s f\|_{L^{q_1}} \|g\|_{L^{r_1}}+\| f\|_{L^{q_2}} \|D^s g\|_{L^{r_2}})
\end{equation}
Another useful result, is the diamagnetic inequality\footnote{The actual diamagnetic unequality allows for more general, so called  magnetic gradients in the form $\nabla_A=\{(\p_j+i A_j)\}_{j=1}^d$ instead of $\nabla$}, which states that for every Schwartz function $f$, there is the point-wise bound a.e. in $x$, 
\begin{equation}
	\label{210} 
	|\nabla |f|(x)|\leq |\nabla f(x)|.
\end{equation}
In particular, if $f\in H^1(\rd)$, then so is $|f|$ and $\| \nabla |f| \|\leq \| \nabla f\|$. 
An important fact, that will be useful in the sequel, is the Caffarelli-Kohn-Nirenberg  inequality, which we discuss next. 
\subsection{Caffarelli-Kohn-Nirenberg inequality}
The following theorem appears in the celebrated work \cite{CKN}. 
\begin{theorem}
	Let $p,q,r,a,b,\ga,\sigma$ and $\theta$
	be real numbers that satisfy, $1\leq p,q$, $0<r$, $0\leq\theta\leq1,$ and
	$$
	\frac{1}{p}+\frac{a}{d}>0, \ \frac{1}{q}+\frac{b}{d}>0, \ \frac{1}{r}+\frac{c}{d}>0\text{ and }c=\theta\sigma+(1-\theta)b,
	$$
	then there exist positive constant $C>0$ such that for all Schwartz functions $u:\rd\to \rone$, there is 
	\begin{equation}
		\label{10} 
		\||x|^{c}u\|_{L^{r}}\leq C\||x|^{a}\nabla u\|_{L^{p}}^{\theta}\||x|^{b}u\|_{L^{q}}^{(1-\theta)},
	\end{equation}
	if and only if the following relation holds,
	\begin{equation}
		\label{20} 
		\begin{cases}
			& 	\frac{1}{r}+\frac{c}{d}=\theta\left(\frac{1}{p}+\frac{a-1}{d}\right)+(1-\theta)\left(\frac{1}{q}+\frac{b}{d}\right) \\
			&  0\leq a-\si  \ \ \textup{if} \ \ \theta>0 \ \ \\
			& a-\si \leq 1, \ \ \textup{if} \ \ \theta>0, \ \ \frac{1}{r}+\frac{c}{d}=\frac{1}{p}+\frac{a-1}{d}. 
		\end{cases}
	\end{equation}	
	In particular, for 
	\begin{equation}
		\label{25} 
		0\leq a<1, \  \ 0\leq\theta\leq1,\ \  \frac{1}{q}=\frac{1}{2}-\frac{\theta}{d}\left(1-a\right), 
	\end{equation}
	\begin{equation}
		\label{30} 
		\left(\int_{\mathbb {R}^{d}}|u|^{q}dx\right)^{1/q}\leq C\left(\int_{\mathbb {R}^{d}}|\nabla u|^{2}|x|^{2a}dx\right)^{\theta/2}\left(\int_{\mathbb {R}^{d}}|u|^{2}dx\right)^{(1-\theta)/2}. 
	\end{equation}
\end{theorem}
{\bf Remarks:} 
\begin{itemize}
	\item The equality relation in  \eqref{20} is nothing but a dimensionality condition, which is clearly necessary for \eqref{10} to hold true. 
	\item In the particular case of $a=0$, we obtain the Gagliardo-Nirenberg inequality. 
\end{itemize}
We introduce the Hilbert space of functions $H^{1,a}, 0\leq a<1$ as follows 
\begin{equation}
	\label{h1a} 
	\|u\|_{H^{1,a}(\rd)}:=\left(\int_{\rd}(|\nabla u|^{2}|x|^{2a}+|u|^{2})\ dx\right)^{1/2}.
\end{equation}

\subsection{Necessary and sufficient conditions for stability of the eigenvalue problem \eqref{l:20}}
\label{sec:2.4}

We now  give a necessary and suffcient condition regarding the spectral stability of the waves $e^{i \om t} \phi$. One is interested in the eigenvalue problem \eqref{l:20} and specifically in the presence of non-trivial solutions $(\la, \vec{v}), \Re \la>0, v\neq 0$. The problem has been well-studied over the last thirty years, and there is a well-developed theory, which culminated in the work \cite{LZ}, but it was preceeded by some earlier works.  Below, we mostly follow the presentation in \cite{LZ}. 

For eigenvalue problems in the form 
\begin{equation}
	\label{k:10} 
	\ci \ck f=\la f, 
\end{equation}
we assume that $\ck: \cx\to \cx^*$ is a bounded operator for appropriate Hilbert space $\cx$, while $\ck=\ck^*$ for a Hilbert space $H: \cx\subset H\subset \cx^*$. Assume also that $\ck$ has finitely many negative eigenvalues. That is, its Morse index, denoted $n(\ck)$ is finite. Regarding the operator $\ci$, we can assume in our case that it is a bounded skew-symmetric operator,$\ci^*=-\ci$. Introduce the integers $k_r$ to be the number of positive eigenvalues of \eqref{k:10}, counting algebraic multiplicities, while $k_c$ is the number of quadruplets of eigenvalues of \eqref{k:10}, while finally $k_i^-$ be the number of purely imaginary eigenvalues with negative Krein signatures. Recall that a pair of imaginary e-values, $\pm i \mu$ (with e-vector $\vec{z}: \ci \ck \vec{z}= i \mu \vec{z}$), has a negative Krein signature, if 
$\dpr{\ck \vec{z}}{\vec{z}}<0$. 

Next,  the generalized kernel is defined in a standard way via
$$
gKer(\mathcal{I}\mathcal{K})=span[(Ker[(\mathcal{I}\mathcal{K})^l], l=1, 2, \ldots]  
$$
Clearly $Ker(\ck)\subset gKer(\mathcal{I}\mathcal{K})$. Assume further that there are finally many vectors $\{\eta_j\}_{j=1}^N$, so that they complement $Ker(\ck)$ to $gKer(\mathcal{I}\mathcal{K})$, i.e. 
$$
gKer(\mathcal{I}\mathcal{K})=span\{\eta_j\}_{j=1}^N \oplus Ker(\ck)
$$ 
Introduce a symmetric matrix $\cd\in \cm_{N\times N}$, $\cd_{i j}:=\dpr{\ck \eta_i}{\eta_j}$. With these preparations, we can state the Hamiltonian index counting theorem, \cite{LZ}, 
 \begin{equation}
	\label{k:20}
	k_{Ham}:=k_r+2 k_c+2k_i^-=n(\ck)-n_0(\cd),
\end{equation}
 where $n_0(\cd)=\#\{\la\leq 0: \la\in \si(\cd)\}$ is the number of non-positive eigenvalues of $\cd$. An easy consequence of parity considerations is that if $n(\ck)=1$, then stability is equivalent to $n_0(\cd)=1$, while instability is the same as $n_0(\cd)=0$. 
  
\section{Variational construction}
We need a preparatory result. 
\subsection{The compactness of the embedding $H^{1,a}(\rd)\hookrightarrow L^q(\rd)$}
There are many weighted Sobolev embedding available in the literature, see for example \cite{SW}. In addition however, we shall need a compactness result for this embedding, which we prefer to state precisely as follows.  
\begin{proposition}
	\label{prop:10} 
	Let $d\geq 1$, and $a\in (0,1), \theta\in (0,1)$, so that $q$ satisfies \eqref{25}. Then the embedding $H^{1,a}(\rd)\hookrightarrow L^q(\rd)$ is compact. 
\end{proposition}
We provide an independent and direct  proof of this result, for the readers convenience, in the appendix. 
We are now ready to proceed with the existence proof for the waves. 
\subsection{Existence of the waves}
Introduce the Weinstein functional 
\begin{equation}
	\label{150}
	J[u]=\frac{\int|\nabla u|^{2}|x|^{2a}dx+\int|u|^{2}dx}{\|u\|_{L^{p+1}}^{2}}
\end{equation}
Note that the Weinstein's functional is homogeneous of degree zero, i.e. $J[m u]=J[u]$ for all $m\neq 0$. 
If the parameters $a, \theta, q$ satisfy the relation \eqref{25}, the Caffarelli-Kohn-Nirenberg's inequality \eqref{30} guarantees that $J[u]\geq C>0$. It follows that the variational problem
\begin{equation}
	\label{160}
	\min_{u:\rd\to\cc, u\neq 0} J[u]
\end{equation}
is well-posed.  We explicitly note that in the variational  problem \eqref{160}, the function $u$ is complex valued.  In the sequel, we show that \eqref{160} has a solution, that is a positive  function $\Phi$, so that $J[\Phi]=\min_{u\neq 0} J[u]$. 
\begin{proposition}
	\label{prop:45} 
	Let $d\geq 1$ and $a, \theta, q=p+1 $ satisfy \eqref{25}. That is, 
	$$
	0\leq a<1, \theta\in (0,1), \f{1}{p+1} = \f{1}{2}  - \f{\theta}{d} (1-a)
	$$
	Then, the variational problem \eqref{160} has a solution $\Phi$. Moreover, $\Phi\geq 0$ and it satisfies the Euler-Lagrange equation 
	\begin{equation}
		\label{170} 
			-\nabla\cdot(|x|^{2a}\nabla\Phi)+\Phi-\ka \Phi^p=0.
	\end{equation}
for some Lagrange multiplier $\ka>0$. 
\end{proposition}
\begin{proof}
	We start with the remark that  by the diamagnetic inequality \eqref{210}, we have that $J[u]\geq J[|u|]$, whence one can consider the variational problem \eqref{160}  as constrained on the set of positive functions. 
	By the homogeneity of $J$, solving \eqref{160} is equivalent to solving the constrained maximization problem 
	$$
	\begin{cases}
		\|u\|_{L^{p+1}}^2\to \max \\
		\|u\|_{H^{1,a}}=1.
	\end{cases}
	$$
	It follows by the remark above that again, we can focus our attention to the case $u\geq 0$. 
	The supremum $Q=\sup_{\|u\|_{H^{1,a}}=1}  \|u\|_{L^{p+1}}^2$ is finite - in fact
	$$
	J_{\min}=\inf_{u\neq 0} \frac{\int|\nabla u|^{2}|x|^{2a}dx+\int|u|^{2}dx}{\|u\|_{L^{p+1}}^{2}}=\f{1}{Q}.
	$$
	so take a maximizing sequence, i.e. $u_n>0 : \|u_n\|_{H^{1,a}}=1, \lim_n \|u_n\|_{L^{p+1}}^2=Q$. Since by Proposition \ref{prop:10},  $H^{1,a}$ is compactly embeded into $L^{p+1}$, we can select a subsequence of $\{u_n\}$, denoted the same, which converges in $L^{p+1}$ to say $\Phi\geq 0$, and at the same time, by Banach-Alaoglu's theorem, it converges weakly in $H^{1,a}$ to the same function. By the lower semi-continuity of weak convergence, we have 
	$
	1=\liminf_n \|u_n\|_{H^{1,a}}\geq \|\Phi\|_{H^{1,a}}, 
	$
	while $Q=\lim_n \|u_n\|_{L^{p+1}}^2=\|\Phi\|_{L^{p+1}}^2$. Since 
	\begin{equation}
		\label{300} 
			\f{1}{Q} = J_{\min} \leq J[\Phi]=\f{\|\Phi\|_{H^{1,a}}^2}{\|\Phi\|_{L^{p+1}}^2} \leq \f{1}{Q}, 
	\end{equation}
	it follows that there are equalities everywhere in \eqref{300}, whence $\|\Phi\|_{H^{1,a}}=1$.  Thus, $\Phi$ is a minimizer for \eqref{160}. Incidentally, it is a standard result in functional analysis, that since  $u_n$ tends weakly in a Hilbert space to $\Phi$ and the norms converge, then it is the case that  $\lim_n \|u_n-\Phi\|_{H^{1,a}}=0$ as well. 
	
	We now pass to the derivation of the Euler-Lagrange equation and the spectral properties of $L_\pm$. We use the fact that $\Phi\geq 0$ is a minimizer of $J[u]$, whence  $m(\eps):=J[\Phi+\eps h]\geq J[\Phi]=m(0)$ for every $\eps\in\rone$ and a real-valued test function $h$. Introduce for simplicity $\la:=\int \Phi^{p+1}(x) dx$. 
	
	We compute,
	\begin{eqnarray*}
	& & 	\int|x|^{2a}|\nabla \Phi+\epsilon\nabla h|^{2}=\int|x|^{2a}|\nabla \Phi|^{2}+2\epsilon\int|x|^{2a}\nabla \Phi\cdot\nabla h+\epsilon^{2}\int|x|^{2a}|\nabla h|^{2}+\mathcal{O}(\epsilon^{3})\\ 
	& & 	\int|\Phi+\epsilon h|^{2}=\int|\Phi|^{2}+2\epsilon\int \Phi h+\epsilon^{2}\int|h|^{2}+\mathcal{O}(\epsilon^{3}) \\ 
	& & \frac{1}{\|\Phi+\epsilon h\|_{p+1}^{2}}= 
	 \lambda^{-2/(p+1)}\left[1-\frac{2}{\lambda}\epsilon\int \Phi^{p}h-\epsilon^{2}\frac{p}{\lambda}\int \Phi^{p-1}h^{2}+\epsilon^{2}\frac{p+3}{\lambda^{2}}\left(\int \Phi^p h  \right)^{2}\right]+\mathcal{O}(\epsilon^{3}) 
	\end{eqnarray*}
	Combining these together, we get the expansion 
	\begin{align*}
		J[\Phi+\epsilon h]= & \lambda^{-\frac{2}{p+1}}\left(\int|x|^{2a}|\nabla \Phi|^{2}+\Phi^{2}\ dx\right)\\
		& +\ \frac{2\epsilon}{\lambda^{\frac{2}{p+1}}}\left(\int|x|^{2a}\nabla \Phi\cdot\nabla h+\Phi h\ dx-\frac{1}{\lambda}\left[\int|x|^{2a}|\nabla \Phi|^{2}+\Phi^{2}\ dx\right]\int \Phi^{p} h\ dx\right)\\
		& +\ \lambda^{-\frac{2}{p+1}}\epsilon^{2}\left(\int|x|^{2a}|\nabla h|^{2}+\int|h|^{2}\right)\\
		& -\ \lambda^{-\frac{2}{p+1}}\epsilon^{2}\frac{4}{\lambda}
		\left(\int \Phi^{p} h\ dx\right)\left[\int|x|^{2a}\nabla \Phi\cdot\nabla h+\Phi h\ dx\right]\\
		& +\ \lambda^{-\frac{2}{p+1}}\epsilon^{2}\left(\int|x|^{2a}|\nabla \Phi|^{2}+\Phi^{2}\ dx\right)\left[-\frac{p}{\lambda}\int \Phi^{p-1}h^{2}+\frac{p+3}{\lambda^{2}}(\int \Phi^p h)^{2}\right]+\mathcal{O}(\epsilon^{3}).
	\end{align*}
	Since $m(0)=\min_{\eps} m(\eps)$, it follows that $m'(0)=0, m''(0)\geq 0$. Selecting
	$$\kappa:=\frac{1}{\lambda}\left[\int|x|^{2a}|\nabla\Phi|^{2}+\Phi^{2}\ dx\right]>0,
	$$
	 we see that these relations for $m$ translate to 
  to the following equation and inequality, 
	\begin{equation}
		\label{320} 
	\dpr{-\nabla\cdot(|x|^{2a}\nabla\Phi)+\Phi-\kappa \Phi^{p}}{h} =   \int|x|^{2a}\nabla\Phi\cdot\nabla h+\Phi h-\kappa\Phi^{p} h\ dx=0
	\end{equation}
		 
	\begin{equation}
		\label{310} 
		\int|x|^{2a}|\nabla h|^{2}+|h|^{2}-\kappa p \Phi^{p-1}h^{2}\ dx+\left[\frac{(p-1)\kappa}{\lambda}\right]\left(\int \Phi^{p} h\ dx\right)^{2}\geq 0.
	\end{equation}
	Clearly, \eqref{320} yields that $\Phi$ satisfies the Euler-Lagrange equation \eqref{170} in distributional sense, while the inequality \eqref{310} will be useful in the sequel.  
\end{proof}

\subsection{The existence claim in Theorem \ref{theo:10}}

An obvious variation of the constrained minimizers yields solutions $\phi$ of \eqref{120}. Indeed, setting up $ \phi(x):=\ka^{\f{1}{p-1}}\Phi(x)$, we obtain a solution of \eqref{120}, given that $\Phi$ solves \eqref{170}. This proves the existence claims in Theorem \ref{theo:10}. The pointwise decay \eqref{18} requires an additional argument and it will be addressed later. 
\section{Stability of the waves}
We start with establishing some basic spectral properties of the linearized operators. 
\subsection{The linearized operators $\cl_\pm$ and their spectral properties}
We now proceed to define the appropriate linearized operators. Informally, 
\begin{eqnarray*}
	L_- f (x)&=& -\nabla\cdot(|x|^{2a}\nabla f)+1-\ka \Phi^{p-1} f\\
	L_+ f (x) &=&  -\nabla\cdot(|x|^{2a}\nabla f)+1-p \ka  \Phi^{p-1} f
\end{eqnarray*} 
In order to introduce them in a formal manner, consider the corresponding symmetric quadratic forms 
\begin{eqnarray*}
	q_-(f,g) &=& \int_{\rd}  |x|^{2a}\nabla f (x)\cdot \nabla \bar{g} (x)+f \bar{g} -\ka \Phi^{p-1} f\bar{g}\\
 q_+ (f,g)&=&  \int_{\rd} |x|^{2a}\nabla f (x)\cdot \nabla \bar{g} (x)+f \bar{g} -p \ka  \Phi^{p-1} f\bar{g} dx 
\end{eqnarray*} 
with\footnote{keeping in mind that $\Phi\in L^\infty$} $D(q_\pm)=H^{1,a}$.  Then, we define the self-adjoint operators $L_\pm$ as the Friedrichs extensions of $q_\pm$.  These self-adjoint operators have respective domains $D(L_\pm)$, which will be the basis of our considerations henceforth.  It is immediately clear from \eqref{320} that $L_-[\Phi]=0$.  We shall need further properties of $L_\pm$, which we state and prove below. 
\begin{proposition}
	\label{prop:56} 
	The self-adjoint operators $L_\pm$ have the following spectral properties
	\begin{itemize}
		\item $L_-[\Phi]=0$, $0$ is a simple eigenvalue and $L_-\geq 0$.  In fact, there is a spectral gap 
		\begin{equation}
			\label{l:35} 
			L_-|_{\{\Phi\}^\perp}\geq \de>0.
		\end{equation}
		
		\item $L_+$ has exactly one negative eigenvalue. 
	\end{itemize}
\end{proposition}
\begin{proof}
	We start with the proof of the properties of $L_+$. The equation \eqref{310} shows that for real-valued test functions $h$, 
	$$
	\langle L_{+}h,h\rangle+
	\left[\frac{(p-1)\kappa}{\lambda}\right]\left(\int\Phi^{p} h\ dx\right)^{2}\geq0.
	$$
	Further, if we take $h\perp \Phi^p$, $\langle L_{+}h,h\rangle\geq0,$so
	we have that
	$
	L_{+}|_{\{\Phi^p\}^{\perp}}\geq0.
	$
	In particular, from the min-max characterization of the eigenvalues for a self-adjoint operator, it follows that $L_+$ has at most one negative eigenvalue. On the other hand, a direct calculation and the Euler-Lagrange equation \eqref{170} show that 
	$$
	\dpr{L_+ \Phi}{\Phi} = -\ka(p-1) \int \Phi^p<0,
	$$
	whence $L_+$ has exactly one negative eigenvalue.

	We now turn our attention to $L_-$. A simple application of the Cauchy-Schwartz inequality shows that,
	\begin{equation}
		\label{330} 
		\left(\int \Phi^p  h\ dx\right)^{2}\leq\left(\int \Phi^{p+1}\ dx\right)\left(\int \Phi^{p-1} h^{2}\ dx\right)=\la \left(\int \Phi^{p-1} h^{2}\ dx\right)
	\end{equation}
	so applying this inequality in  \eqref{310} yields the inequality 
	\begin{eqnarray*}
		& & 0\leq 	\int|x|^{2a}|\nabla h|^{2}+|h|^{2}-\kappa p \Phi^{p-1}h^{2}\ dx+\left[\frac{(p-1)\kappa}{\lambda}\right]\left(\int \Phi^{p} h\ dx\right)^{2} \\
		&\leq& \int|x|^{2a}|\nabla h|^{2}+h^{2}-\kappa \Phi^{p-1} h^{2}\ dx=\langle L_{-}h,h\rangle 
	\end{eqnarray*}
	for all real-valued $h$. As $L_-$ preserves real and odd parts, it follows that $L_-\geq 0$. 
	We need however further properties of $L_-$, namely the spectral gap at zero. 
	To this end, note the formula for real-valued functions $f,g$, 
	\begin{equation}
		\label{400} 
	\int f^2 \int g^2 - \left(\int f g\right)^2 = 	\|\mu f+g\|^2_{L^2} \int f^2, \ \ \mu=-\f{\int f g}{\int f^2}. 
	\end{equation}
	which yields a quantitative form of the positive remainder in the Cauchy-Schwartz's inequality. We now show the spectral gap property for $L_-$.
	
	Suppose for a contradicition that the spectral gap for $L_-$  does not hold. Since we have already established $L_-\geq 0$, it must be that there is a sequence $h_n\in D(L_-): h_n\perp \Phi, \|h_n\|=1$, so that $\lim_n \dpr{L_- h_n}{h_n}=0$. Apply \eqref{400} to 
	$$
	f=\Phi^{\f{p+1}{2}}, g_n=h_n \Phi^{\f{p-1}{2}}, 
	\mu_n=-\f{\int h_n \Phi^p}{\int \Phi^{p+1}}.
	$$
	 We obtain, 
	\begin{equation}
		\label{410} 
		\int \Phi^{p+1}  \int h_n^2 \Phi^{p-1} - \left(\int h_n \Phi^p\right)^2 = 	\|h_n \Phi^{\f{p-1}{2}} - \mu_n \Phi^{\f{p+1}{2}}\|^2_{L^2}  	\int \Phi^{p+1}. 
	\end{equation}
	In the notations of the Proposition \ref{prop:45}, note that we have introduced $\la=\int \Phi^{p+1}$. Denoting the quadratic form in \eqref{310} by $q: q(h,h)\geq 0$, we have the relation 
	\begin{eqnarray*}
		0 &\leq & q(h_n,h_n)=\int|x|^{2a}|\nabla h_n|^{2}+|h_n|^{2}-\kappa p \Phi^{p-1}h_n^{2}\ dx+\left[\frac{(p-1)\kappa}{\lambda}\right]\left(\int \Phi^{p} h_n\ dx\right)^{2}=\\
		& =& \int|x|^{2a}|\nabla h_n|^{2}+|h_n|^{2}-\kappa  \Phi^{p-1}h_n^{2}\ dx -\f{\ka(p-1)}{\la} \left[\la \int h_n^2 \Phi^{p-1} - \left(\int h_n \Phi^p\right)^2\right] \\ 
		&=&\dpr{L_- h_n}{h_n}-\ka(p-1) \|h_n \Phi^{\f{p-1}{2}} - \mu_n \Phi^{\f{p+1}{2}}\|^2_{L^2}. 
	\end{eqnarray*}
	Rearranging terms yields 
	$$
	\dpr{L_- h_n}{h_n}\geq \ka(p-1) \|h_n \Phi^{\f{p-1}{2}} - \mu_n \Phi^{\f{p+1}{2}}\|^2_{L^2}. 
	$$
	As $\lim_n \dpr{L_- h_n}{h_n}=0$, it follows that $\lim_n \|h_n \Phi^{\f{p-1}{2}} - \mu_n \Phi^{\f{p+1}{2}}\|_{L^2}=0$. 
	
By Cauchy-Schwartz, $\mu_n\leq \f{\|h_n\| \|\Phi\|_{L^{2p}}^p}{\la}$ is a bounded sequence. By passing to a subsequence if necessary, we may  and do assume that $\mu_n$ converges,  to say $\mu_0$. If $\mu_0=0$, we conclude that $\lim_n \|h_n \Phi^{\f{p-1}{2}}\|_{L^2}=0$, which together with $\lim_n \dpr{L_- h_n}{h_n}=0$ implies that $\lim_n \|h_n\|_{H^{1,a}}=0$, a contradiction. If $\mu_0\neq 0$, we conclude that 
$$
\lim_n \int \Phi^{p-1} |h_n-\mu_0 \Phi|^2 dx=0.
$$
	Consider the open set $\ca=\{x\in \rd: \Phi(x)>0\}$. For every test function, $\zeta\in C^\infty_0(\ca)$, we have that 
	$$
	|\dpr{\zeta}{h_n-\mu_0 \Phi}|^2\leq \|\zeta\|^2_{L^2} \int_{supp \zeta} |h_n-\mu_0 \Phi|^2 \leq c_\zeta \int \Phi^{p-1} |h_n-\mu_0 \Phi|^2 dx,
	$$
	whence $\lim_n \dpr{\zeta}{h_n-\mu_0 \Phi}=0$, and so $h_n-\mu_0\Phi$ tends  weakly to zero,  in $L^2(\ca)$. In particular, testing this weak convergence with $\Phi$, we obtain (using $h_n\perp \Phi$), $0=\lim_n  \dpr{\Phi}{h_n-\mu_0 \Phi}=-\mu_0 \|\Phi\|^2\neq 0$, a contradiction. All in all, we have established the spectral gap property. 
\end{proof}
Clearly, rescaling $\Phi\to \phi$ reveals the appropriate properties of the operators $\cl_\pm$. We summarize them in the following Proposition. 
\begin{proposition}
	\label{prop:52} 
	Let $\phi:=\ka^{\f{1}{p-1}} \Phi$, where $\Phi$ is the minimizer of \eqref{150} from Proposition \ref{prop:45}.  Then, $\phi$ satisfies \eqref{120}. In addition, the corresponding  linearized operators $\cl_\pm$ satisfy   
	\begin{itemize}
		\item $\cl_-[\phi]=0$, $0$ is a simple eigenvalue and $\cl_-\geq 0$.  There is also the spectral gap property 
		\begin{equation}
			\label{l:305} 
			\cl_-|_{\{\phi\}^\perp}\geq \de>0.
		\end{equation}
		\item $\cl_+$ has exactly one negative eigenvalue. 
	\end{itemize}
\end{proposition}
\begin{proof}
	A consequence of Proposition \ref{prop:56}. 
\end{proof}
We are now ready to proceed with the stability of the waves $e^{i \om t}\phi$. 
\subsection{Spectral stability analysis of the waves $e^{i \om t}\phi$}
Specifically, in view of Proposition \ref{prop:52}, we have that $\cl$ is self-adjoint, with domain $D(\cl_+)\times D(\cl_-)$ of the Friedrich's extensions. In the setup of the eigenvalue problem \eqref{k:10}, $\ci:=\cj$, $\ck:=\cl$, and the Hilbert space space $X=H^{1,a}(\rd)$  and its Banach space dual,  $X^*=H^{-1,a}$, while $H:=L^2(\rd)$. 

Next, we look into $Ker(\cl)=\left(\begin{array}{c}
Ker(\cl_+) \\ Ker(\cl_-)
\end{array}\right)$. We have already determined, see Proposition \ref{prop:52} that $Ker(\cl_-)=span[\phi]$, while the question for $Ker(\cl_+)$ was left open\footnote{although we have of course conjectured that $Ker(\cl_+)=\{0\}$. }. Let us proceed however toward a the study of  $gKer(\cl)$ anyway. 

We first look into adjoint eigenvectors behind $\left(\begin{array}{c}
0 \\ \phi
\end{array}\right)\in Ker(\cl)$. This means, by applying $\cj^{-1}=-\cj$, we solve 
$$
\cl \vec{v}=\left(\begin{array}{c}
	 \phi \\ 0
\end{array}\right)\in Ker(\cl). 
$$
Provided $\phi\in Im(\cl_+)$, we can solve this uniquely  $v_1=\cl_+^{-1}\phi \in Ker(\cl_+)^\perp$, so $\vec{v}=\left(\begin{array}{c}
	\cl_+^{-1}\phi \\ 0
\end{array}\right)$. 
Looking further for second adjoint vector, we need to solve 
$$
\cl \vec{z}=\left(\begin{array}{c}
0 \\ -	\cl_+^{-1}\phi 
\end{array}\right),
$$
or equivalently $\cl_-z_2=-	\cl_+^{-1}\phi $. This should satisfy the solvability/Fredholm condition \\ $\dpr{\cl_+^{-1}\phi}{\phi}=0$. Thus, whenever $\dpr{\cl_+^{-1}\phi}{\phi}\neq 0$, we have exactly only one adjoint eigenvector behind $\left(\begin{array}{c}
	0 \\ \phi
\end{array}\right)$, namely $\left(\begin{array}{c}
\cl_+^{-1}\phi \\ 0
\end{array}\right)$. 

Let now assume that some $\psi\in Ker(\cl_+)$ and we look for generalized eigenvectors behind $\left(\begin{array}{c}
\psi \\ 0
\end{array}\right)$. Matters reduce to solving 
$$
\cl \vec{v}=\left(\begin{array}{c}
	0 \\ -\psi
\end{array}\right). 
$$
Thus, subject to the solvability condition $\dpr{\psi}{\phi}=0$, we can solve uniquely $v_2=\cl_-^{-1} \psi\in \{\phi\}^\perp$. Looking for further adjoints means we solve 
$$
\cl \vec{z}=\left(\begin{array}{c}
	\cl_-^{-1}\psi \\ 0
\end{array}\right).
$$
This  last equation requires  the solvability condition, $\dpr{\cl_-^{-1} \psi}{\psi}=0$. Recall however that by Proposition \ref{prop:52}, and specifically \eqref{l:305}, we have that $\cl_-$ is strictly non-negative on $\{\phi\}^\perp$, whence $	\cl_-^{-1}|_{\{\phi\}^\perp}>0$. It follows that the Fredholm condition $\dpr{\cl_-^{-1} \psi}{\psi}=0$ fails, and there is no additional adjoint eigenvectors, behind $\left(\begin{array}{c}
	\psi \\ 0
\end{array}\right)$ (which itself may not exist). Thus, we have shown that the system 
$$
\eta^{(1)}:=\left(\begin{array}{c}
	\cl_+^{-1}\phi \\ 0
\end{array}\right), \eta^{(2)}:=\left(\begin{array}{c}
0 \\ \cl_-^{-1}\psi \
\end{array}\right), 
$$
spans $gKer(\cj\cl)$, provided $\dpr{\cl_+^{-1}\phi}{\phi}\neq 0$. 
We have thus established the following proposition. 
\begin{proposition}
	\label{prop:39}
Let $\cl_\pm$ be the linearized operators , corresponding of the  waves $\phi$. In addition to the properties listed in Proposition \ref{prop:52}, they do satisfy the following spectral properties, provided $\phi\in Ran \cl_+=Ker[\cl_+]^\perp$ and $\dpr{\cl_+^{-1}\phi}{\phi}\neq 0$. 
\begin{itemize}
	\item $gKer(\cj\cl)$ is at most two dimensional and 
	$$
	gKer(\cj\cl)=span[\eta^{(1)}, \eta^{(2)}].
	$$
\end{itemize}
\end{proposition}
It remains to analyze the matrix $\cd$. 
In forming the matrix $\cd$, as described in Section \ref{sec:2.4}, we observe that 
$$
\cd=\left(\begin{array}{cc}
	\dpr{\cl_+^{-1} \phi}{\phi} & 0 \\
	0 & 	\dpr{\cl_-^{-1} \psi}{\psi}
\end{array}\right).
$$
Note that as $\psi\perp \phi$ and $	\cl_-^{-1}|_{\{\phi\}^\perp}>0$, we have that $\dpr{\cl_-^{-1} \psi}{\psi}>0$, whence $n_0(\cd)=n_0(	\dpr{\cl_+^{-1} \phi}{\phi})$. So, it remains to compute the sign of this quantity. 

To this end, note that as $\phi_\om$ satisfy the equation \eqref{120}, then by a simple scaling argument, we see that 
\begin{equation}
	\label{380} 
	\phi_\om(x)=\om^{\f{1}{p-1}} \phi_1(\om^{\f{1}{2(1-a)}}x)
\end{equation}
where $\phi_1$ solves \eqref{120} with $\om=1$. Furthermore, taking into account the formula \eqref{380}, $\om\to \phi_\om$ is smooth as a mapping from $\rone_+\to H^{2,a}(\rd)$, so one can take a derivative with respect to $\om$ in \eqref{120}. The result is 
$$
\cl_+[\p_\om ]=-\phi_\om.
$$
Thus, $\phi_\om \in Ran[\cl_+]$ and moreover, $\p_\om\phi_\om=-\cl_+^{-1}[\phi]$. This allows us to compute the quantity of interest, namely $\dpr{\cl_+^{-1} \phi}{\phi})$. We obtain 
$$
\dpr{\cl_+^{-1} \phi}{\phi})=-\f{1}{2}\p_\om \|\phi_\om\|^2_{L^2}=-\f{1}{2}\left(\f{2}{p-1}-\f{d}{2(1-a)}\right) \om^{\f{2}{p-1}-\f{d}{2(1-a)}-1}
\|\phi_1\|^2_{L^2},
$$
where we have taken advantage of the formula \eqref{380}. Thus, the stability is equivalent to $\left(\f{d}{2(1-a)}-\f{2}{p-1}\right)<0$, or equivalently, 
$$
1<p<1+\f{4(1-a)}{d}.
$$
Clearly, the instability occurs at $p>1+\f{4(1-a)}{d}$, while the bifurcation point $p=1+\f{4(1-a)}{d}$ features an extra pair of eigenvalues at zero, as discussed earlier.  
 
 This completes the proof of Theorem \ref{theo:20}. 
  
\section{Finite time blow-up for the spectrally unstable waves}
We begin the proof of Theorem \ref{theo:30}. Assuming, for a contradcition, that $u$ is a global solution to \eqref{14},  multiply  it by $\bar{u}$ and take
the imaginary part. The result is 
\begin{equation}
	\label{600}
	\frac{d}{dt}|u|^{2}=-\nabla\cdot(2|x|^{2a}\mathfrak{I}(\bar{u}\nabla u)).
\end{equation}
Then multiplying equation \eqref{14} by $|x|^{2-2a}$
and integrating by parts, we get 
\begin{align}
	\frac{d}{dt}\int|x|^{2-2a}|u|^{2}\ dx= & -\int|x|^{2-2a}\nabla\cdot(2|x|^{2a}\mathfrak{I}(\bar{u}\nabla u))\ dx\label{605}\\
	= & \ 4(1-a)\sum_{j}\int x_{j}\mathfrak{I}(\bar{u}u_{j})\ dx.\nonumber 
\end{align}
Now let us multiply \eqref{14} by $x\cdot\nabla\bar{u}$,
use integration by parts and take the real parts. We shall analyze
term-wise. With the non linear term, we get the relation, 
$$
\mathfrak{R}\left(\int|u|^{p-1}u(x\cdot\nabla\bar{u})\right)=-\frac{d}{p+1}\int|u|^{p+1}\ dx,
$$
while for the dispersion term,  we get
$$
\mathfrak{R}\left(\int\nabla\cdot(|x|^{2a}\nabla u)(x\cdot\nabla\bar{u})\ dx\right)=\frac{d+2a-2}{2}\int|x|^{2a}|\nabla u|^{2}\ dx.
$$
For the last term observe that 
\begin{eqnarray*}
	\mathfrak{R}\left(\int iu_{t}(x\cdot\nabla\bar{u})\ dx\right) &= & \ \mathfrak{R}\left(i\int\sum_{j}x_{j}\bar{u_{j}}u_{t}\ dx\right) 
	= \ \mathfrak{R}\left(\frac{i}{2}\sum_{j}\int x_{j}[\bar{u_{j}}u_{t}-u_{j}\bar{u_{t}}]\ dx\right)\\
	&= & \ \mathfrak{R}\left(\frac{i}{2}\sum_{j}\int x_{j}[\frac{\partial}{\partial t}(\bar{u_{j}}u)-\frac{\partial}{\partial j}(u\bar{u_{t}})]\ dx\right)\\
	&= & \ \mathfrak{\frac{\partial}{\partial t}R}\left(\frac{i}{2}\sum_{j}\int x_{j}\bar{u_{j}}u\right)+d\mathfrak{R}\left(\frac{i}{2}\int\bar{u_{t}}u\ dx\right)\\
	& = & \ \frac{1}{2}\frac{\partial}{\partial t}\sum_{j}\int x_{j}\mathfrak{I}(\bar{u}u_{j})\ dx+d/2\left(-\int|x|^{2a}|\nabla u|^{2}+\int|u|^{p+1}\ dx\right)
\end{eqnarray*}
Therfore we get, 
\begin{align}
	\label{610}
	\frac{1}{2}\frac{d}{dt}\sum_{j}\int x_{j}\mathfrak{I}(\bar{u}u_{j})\ dx= & (1-a)\int|x|^{2a}|\nabla u|^{2}\ dx+d(\frac{1}{p+1}-\frac{1}{2})\int|u|^{p+1}\ dx\\
	= & 2(1-a)E(u)-\frac{1}{2(p+1)}[4a+d(p-1)-4]\int|u|^{p+1}\ dx.
	\nonumber 
\end{align}
Combining equation \eqref{605} and \eqref{610},
we get the virial identity, 
\begin{equation}
	\frac{1}{16(1-a)}\frac{d^{2}}{dt^{2}}\||x|^{(1-a)}u(t)\|_{L^{2}}^{2}=P(u(t))\label{620}
\end{equation}
where
\begin{equation}
	P(u)=\frac{(1-a)}{2}\int|x|^{2a}|\nabla u|^{2}-\frac{\alpha}{2(p+1)}\int|u|^{p+1},\qquad\alpha=\frac{d(p-1)}{2}
	\label{630}
\end{equation}
We follow the arguments by Ohta, \cite{Oh1}   for
the proof of Theorem \ref{theo:30}. 
To this end, let $\phi_{\omega}^{\lambda}(x):=\lambda^{d/2}\phi_{\omega}(\lambda x)$, the
$L^{2}$ norm invariant scaling. Observe that $\|\phi_{\omega}^{\lambda}-\phi_{\omega}\|_{H^{1,a}}\to0$,
as $\lambda\to1^{+}.$ The goal is to show that for the functions
$u_{\lambda}(t,x)$ that solves \eqref{14} with
$u_{0}(x)=\phi_{\omega}^{\lambda}(x),$ 
\begin{equation}
	\label{f:10} 
	\frac{d^{2}}{dt^{2}}\||x|^{(1-a)}u_{\lambda}(t)\|_{L^{2}}^{2}<(1-a)(E(\phi_\om^\la) -E(\phi_{\omega}))
\end{equation}
which implies a finite time blowup, once we show $E(\phi_{\omega}^{\lambda})<E(\phi_{\omega})$. Indeed, \eqref{f:10} implies that the positive function $t\to \||x|^{(1-a)}u_{\lambda}(t)\|_{L^{2}}^{2}$ is concave for all $t>0$, an impossibility.

We complete the remainder of the proof in three 
steps.
\begin{itemize}
	
	\item \emph{Step-1:  $E(u^{\lambda}(t))=E(\phi_{\omega}^{\lambda})<E(\phi_{\omega})$.}
	
	Observe that 
	$$
	E(\phi_{\omega}^{\lambda})=\frac{\lambda^{2-2a}}{2}\int|x|^{2a}|\nabla\phi_{\omega}|^{2}\ dx-\frac{\lambda^{\alpha}}{p+1}\int|\phi_{\omega}|^{p+1}\ dx.
	$$
	
	By the Pohozaev identities \eqref{500}, we get that
	$$
	\int|x|^{2a}|\nabla\phi_{\omega}|^{2}\ dx=\frac{\alpha}{(p+1)(1-a)}\int|\phi_{\omega}|^{p+1}\ dx,
	$$
	whence,  it follows that 
	$$
	E(\phi_{\omega})=\frac{\alpha-2(1-a)}{2(p+1)(1-a)}\|\phi_{\omega}\|_{L^{p+1}}^{p+1}, \quad\text{\ensuremath{\quad E(\phi_{\omega}^{\lambda})=\frac{\alpha\lambda^{2-2a}-2(1-a)\lambda^{\alpha}}{2(p+1)(1-a)}}
		\ensuremath{\|\phi_{\omega}\|_{L^{p+1}}^{p+1}}}.
	$$
	If one considers the function $f(\la)=\alpha\lambda^{2-2a}-2(1-a)\lambda^{\alpha}-\alpha-2(1-a)$,
	then we have $f'(1)=0$ and $f''(1)<0$ precisly when $p>1+\frac{4(1-a)}{d}$.
	Thus,  it follows that $E(\phi_{\omega}^{\lambda})<E(\phi_{\omega})$
	for $\lambda$ sufficiently close to $1.$
	
	Now by the conservation
	of energy, 
	$$
	E(u_{\lambda}(t))=E(\phi_{\omega}^{\lambda})<E(\phi_{\omega}),
	$$
	which completes Step 1. 
	\item \emph{Step-2: $\|u_{\lambda}(t)\|_{L^{p+1}}^{p+1}>\|\phi_{\omega}\|_{L^{p+1}}^{p+1}$} for all $t>0$. 
	
	Observe that $\|u_\la(0)\|_{L^{p+1}}^{p+1}=\|\phi_{\omega}^{\lambda}\|_{L^{p+1}}^{p+1}=\lambda^{\alpha}\|\phi_{\omega}\|_{L^{p+1}}^{p+1}>\|\phi_{\omega}\|_{L^{p+1}}^{p+1}$,
	for $\lambda>1.$ Now suppose for a contradiction, that at some time $t_{0}>0$, 
	$$
	\|u_{\lambda}(t_{0})\|_{L^{p+1}}^{p+1}=\|\phi_{\omega}\|_{L^{p+1}}^{p+1}. 
	$$
	By conservation of mass, $\|u_{\lambda}(t_{0})\|_{L^{2}}=\|\phi_{\omega}\|_{L^{2}}$
	and since $\phi_{\omega}$ is a minimizer for \eqref{150}, we have
	that 
	$$
	\int|x|^{2a}|\nabla\phi_{\omega}|^{2}\ dx\leq\int|x|^{2a}|\nabla u_{\lambda}|^{2}\ dx.
	$$
	It follows that $E(\phi_{\omega})\leq E(u_{\lambda}(t)),$which
	is a contradiction to Step-1. Thus 
	$$
	\|u_{\lambda}(t)\|_{L^{p+1}}^{p+1}>\|\phi_{\omega}\|_{L^{p+1}}^{p+1}
	$$
	for $t>0$. 
	\item \emph{Step-3: Proof of \eqref{f:10}}
	
We have, from energy conservation, 
\begin{eqnarray*}
	& &  \frac{1}{2}\||x|^{a}\nabla u_{\lambda}\|^{2}_{L^2}-\frac{1}{p+1}\|u_{\lambda}(t)\|^{p+1}_{L^{p+1}}= E(u_\la(t))=E(\phi_\om^\la)=  E(\phi_\om^\la) -E(\phi_{\omega})+E(\phi_{\omega})\\
	&=& E(\phi_\om^\la) -E(\phi_{\omega})+\frac{\alpha-2(1-a)}{2(p+1)(1-a)}\|\phi_{\omega}\|_{L^{p+1}}^{p+1}\leq E(\phi_\om^\la) -E(\phi_{\omega})+ \frac{\alpha-2(1-a)}{2(p+1)(1-a)}\|u_{\lambda}(t)\|_{L^{p+1}}^{p+1}, 
\end{eqnarray*}
where in the last step, we have used $\alpha-2(1-a)>0$ and Step -2. Simplifying then leads to 
$$
\frac{1}{2}\||x|^{a}\nabla u_{\lambda}\|^{2}_{L^2}-\frac{\al }{2(p+1)(1-a)}\|u_{\lambda}(t)\|^{p+1}_{L^{p+1}}\leq E(\phi_\om^\la) -E(\phi_{\omega}), 
$$
which is, after multiplying by $(1-a)$, 
$$
P(u_\la(t))\leq (1-a) (E(\phi_\om^\la) -E(\phi_{\omega})),
$$
Taking into account \eqref{620} leads to \eqref{f:10}. 
 
\end{itemize}
This completes the proof of Theorem \ref{theo:30}.

\section{Proof of Theorem \ref{theo:10}} 
We start with a fairly standard observation that variational solutions of \eqref{120} are smooth, at least away from the origin. 
\subsection{Smoothness away from the origin}
We have the following Proposition. 
\begin{proposition}
	\label{prop:16} 
	Let $\phi$ be a distrbutional solution of \eqref{120}, with $\phi\in H^{1,a}$. Then, $\phi\in C^\infty(\rd\setminus\{0\})$. 
\end{proposition}
{\bf Remark:} In here, we estimate various norms of cutoffs of $\phi$. These rough estimates invariably come with a very bad dependence on the cutoff parameters. In this sense, we establish only that the solution is smooth, away from the origin, so that we can operate freely with it, but it does not imply decay as $|x|\to \infty$ or any smoothness at zero whatsoever. 
\begin{proof}
	Introduce a cut-off funciton  $\chi$, which is smooth and supported in $|x|<2$ with $\chi=1$ for
	$|x|<1$.  
	Fix $\eps: 0<\eps<1$, and set  $\eta_\eps(x):=(1-\chi(x/\epsilon)) \chi(\eps x)$, localizes away from zero and infinity. Let $\phi_\eps:=\phi \eta_\eps$. 
	Multiplying throughout \eqref{120} by $\eta_\epsilon$, 
	\begin{equation}
		\label{880} 
	-|x|^{2a}\Delta\phi_{\eps}+\phi_\epsilon=\phi^{p}\eta_\epsilon +2a|x|^{2a-2}\sum_{j}x_{j} \eta_\epsilon \partial_{j}\phi -2|x|^{2a}\sum_{j}\partial_{j}\phi \p_j \eta_\epsilon -|x|^{2a}\phi\Delta \eta_\epsilon
	\end{equation}
We first obtain $L^\infty$ estimates on $\phi_\epsilon$. As $\phi_\eps$ is supported in the compact set 
$|x|\in (\eps, 2 \eps^{-1})$, we can infer from \eqref{880}, 
\begin{equation}
\label{890}
(-\epsilon^{2a} \De +1) \phi_\eps \leq \phi^{p}\eta_\epsilon +2a|x|^{2a-2}\sum_{j}x_{j} \eta_\epsilon \partial_{j}\phi -2|x|^{2a}\sum_{j}\partial_{j}\phi \p_j \eta_\epsilon -|x|^{2a}\phi\Delta \eta_\epsilon
\end{equation}
Applying the resolvent $(-\epsilon^{2a} \De +1)^{-1}$, which has positive kernel, we obtain 
$$
\phi_\eps\leq (-\epsilon^{2a} \De +1)^{-1}\left[ \phi^{p}\eta_\epsilon +2a|x|^{2a-2}\sum_{j}x_{j} \eta_\epsilon \partial_{j}\phi -2|x|^{2a}\sum_{j}\partial_{j}\phi \p_j \eta_\epsilon -|x|^{2a}\phi\Delta \eta_\epsilon\right]. 
$$
Taking absolute values (recall the resolvent $(-\epsilon^{2a} \De +1)^{-1}$ has positive bounded integrable kernel), and taking $L^\infty$ norms, we arrive at 
$$
\|\phi_\epsilon\|_{L^\infty}\leq C_\eps(\|\phi\|_{L^p}^p+\|\phi_\eps\|_{H^1})\leq 
C_\eps (1+\|\phi\|_{H^{1,a}})^p
$$
Thus, we have the $L^\infty$ bounds on $\phi_\eps$. 

Taking $L^1$ norms of both sides of \eqref{880}, and using the {\it a priori} estimates on $\phi$ (recall $\phi_\eps$ is supported in the compact set 
$|x|\in (\eps, 2 \eps^{-1})$, and $\phi\in H^{1,a}\subset L^{p^*}$), we obtain a bound for $\|\Delta \phi_\eps\|_{L^1}\leq C_\eps$. Taking a derivative in \eqref{880} and again $L^1$ norm, by incorporating the $L^\infty$ bounds on $\phi_\eps$ and $\|\Delta \phi_\eps\|_{L^1}\leq C_\eps$, we obtain $\|\nabla^3 \phi_\eps\|_{L^1}\leq C_\eps$ etc. This can be iterated to 
$
\|\nabla^N \phi_\eps\|_{L^1}\leq C_{N,\eps},
$
which implies $\phi_\eps\in C^\infty(\rd)$. As this can be applied to any $\epsilon>0$, it follows that $\phi\in C^\infty(\rd\setminus\{0\})$. 
\end{proof}
 \subsection{Power decay at infinity}
Next, we establish the decay rate at infinity. This goes through a series of bootstrapping arguments, each time improving the decay in various norms. 
\begin{proposition}
	\label{app:pr10} 
	Let $d\geq 1, a\in (0,1), p>1$ satisfy \eqref{527}. Let also $\phi$ be a positive   $C^\infty(|x|>1) \cap H^{1,a}(\rd)$ solution of \eqref{120}.
	Then, for each $k\geq 0$, and for each  integer $N\in \cn$, there exists $C_N$, so that  
	\begin{equation}
		\label{498}
		\int_{|x|\sim 2^k} \phi^2(x) dx + \int_{|x|\sim 2^k} |\nabla \phi|^2(x) dx	 \leq C_N 2^{- k N}
	\end{equation}
	In other words, the $H^1(|x|\sim M)$ norm decays  faster than any polynomial of $M$. 
\end{proposition}
\begin{proof}
	Again, we start with 
	$$
	\|\phi\|_{L^r}\leq C_r \|\phi\|_{H^{1,a}}, \ \ 2\leq r<p^*=\f{2d}{d-2(1-a)}. 
	$$
	Recall 
	$$
	2a|x|^{2a-2}\sum_{j}x_{j}\partial_{j}\phi+|x|^{2a}\Delta\phi-\phi+\phi^{p}=0.
	$$
	Multiplying throughout with a cut of funciton $\chi_{k}:=\chi(x/2^{k})$,
	where $\chi$ is smooth and supported in $|x|<2$ with $\chi=1$ for
	$|x|<1$, and denoting $\phi_k=\phi\chi_{k}$,  we obtain the
	equation
	
	\begin{equation}
		\label{900} 
		\phi_{k}-|x|^{2a}\Delta\phi_{k}=\phi^{p}\chi_{k}+2a|x|^{2a-2}\sum_{j}x_{j}\partial_{j}\phi\chi_{k}-2|x|^{2a}\sum_{j}\partial_{j}\phi(\chi_{k})_{j}-|x|^{2a}\phi\Delta\chi_{k}.
	\end{equation}
	Multiply  \ref{900} by $\phi_{k}$ and integrating
	each term, we get the following 
	\begin{align*}
		-\int|x|^{2a}\Delta\phi_{k}\phi_{k}\ dx & =\int|x|^{2a}|\nabla\phi_{k}|^{2}\ dx+a\sum_{j}\int\partial_{j}(\phi_{k}^{2})x_{j}|x|^{2a-2}\ dx\\
		& =\ \int|x|^{2a}|\nabla\phi_{k}|^{2}\ dx+a\sum_{j}\int\partial_{j}(\phi^{2})x_{j}|x|^{2a-2}\chi_{k}^{2}+a\sum_{j}\int\phi^{2}x_{j}|x|^{2a-2}\partial_{j}(\chi_{k}^{2})
	\end{align*}
	By virtue of  the inequality $|\partial_{j}\chi_{k}|\leq C2^{-k}(\chi_{k-1}+\chi_{k}+\chi_{k+1})$, we get
	\begin{eqnarray*}
		& & 	-\int2|x|^{2a}\sum_{j}\partial_{j}\phi(\chi_{k})_{j}\phi_{k}\ dx  = -\frac{1}{2}\sum_{j}\int|x|^{2a}\partial_{j}(\phi^{2})\partial_{j}(\chi_{k}^{2})=\\
		& = &\frac{1}{2}\sum_{j}\int|x|^{2a}\phi^{2}\partial_{j}^{2}(\chi_{k}^{2})+a\sum_{j}\int x_{j}|x|^{2a-2}\phi^{2}\partial_{j}(\chi_{k}^{2})=\\
		& = & \sum_{j}\int|x|^{2a}\phi^{2}\partial_{j}^{2}(\chi_{k})\chi_{k}+\sum_{j}\int|x|^{2a}\phi^{2}(\partial_{j}\chi_{k})^{2}+a\sum_{j}\int x_{j}|x|^{2a-2}\phi^{2}\partial_{j}(\chi_{k}^{2})
	\end{eqnarray*}
	Combining everything, we get the equation,
	\begin{equation}
		\label{910}
		\int\phi_{k}^{2}\ dx+\int|x|^{2a}|\nabla\phi_{k}|^{2}\ dx=\int\phi^{p}\phi_{k}\chi_{k}\ dx+\int|x|^{2a}\phi^{2}|\nabla\chi_{k}|^{2}\ dx
	\end{equation}
	As $\phi\in H^{1,a}$, we obtain an apriori bound on the left
	hand side of equation \eqref{910} 
	$$
	\int\phi_{k}^{2}\ dx+\int|x|^{2a}|\nabla\phi_{k}|^{2}\ dx\leq C. 
	$$
	In particular, by the support considerations, 
	\begin{equation}
		\label{920} 
		\|\nabla \phi_k\|_{L^2}\leq C 2^{- a k}, \|\phi_k\|_{L^2}\leq C. 
	\end{equation}
	Fix $q: \f{1}{2}-\f{1-a}{d}>\f{1}{q}>\f{d-2}{2d}$ or equivalently $\f{2d}{d-2}>q>\f{2d}{d-2(1-a)}$. Then, one can check 
	$$
	p\f{q}{q-1}< p \f{2d}{d+2(1-a)}\leq \f{d+2(1-a)}{d-2(1-a)} \f{2d}{d+2(1-a)}=p^*.
	$$
	Let $\theta\in (0,1): \theta=d\left(\f{1}{q}-\f{d-2}{2d}\right)$, so that 
	\begin{equation}
		\label{940} 
		\f{1}{q}= \f{\theta}{2}+(1-\theta) \f{d-2}{2d}.
	\end{equation}
	By H\"older's, it follows 
	\begin{equation}
		\int|\phi^{p}\phi_{k}\chi_{k}|\ dx\leq\|\phi^{p}\|_{L^{\f{q}{q-1}}}\|\phi_{k}\|_{L^q}\leq\|\phi\|_{L^{\f{pq}{q-1}}}^p\|\phi_{k}\|_{L^2}^{\theta}\|\phi_{k}\|_{L^{\frac{2d}{d-2}}}^{(1-\theta)}\leq C\|\nabla\phi_{k}\|_{2}^{1-\theta}\lesssim2^{-ka(1-\theta)},\label{eq:inequ1}
	\end{equation}
	where we have used the relation \eqref{940} in conjunction to the Gagliardo-Nirenberg's inequality. Also we have by compact support of smooth
	function $\chi$, 
	\begin{equation}
		\label{950} 
		\int|x|^{2a}\phi^{2}|\nabla\chi_{k}|^{2}\ dx\lesssim2^{-2k(1-a)}
	\end{equation}
	Combining these two inequalities, we get the new bound, 
	\begin{equation}
		\label{960}
		\int\phi_{k}^{2}\ dx+\int|x|^{2a}|\nabla\phi_{k}|^{2}\ dx\lesssim 2^{-k \min(a(1-\theta), 2(1-a))}. 
	\end{equation}
	Continuing in this way, by using the improved version \eqref{960}, instead of \eqref{920}, we get an improvement and so on. In the end, we get
	an inequality, 
	$$
	\int\phi_{k}^{2}\ dx+\int|x|^{2a}|\nabla\phi_{k}|^{2}\ dx\leq C_N  2^{-N k \min(a(1-\theta), 2(1-a))}. 
	$$
\end{proof}
It is immediately clear rom H\"older's that the estimates \eqref{498} can be extended to 
$$
\int_{|x|\sim 2^k} \phi_k+ |\nabla \phi_k| \leq C_N 2^{-kN}.
$$
It is also clear from the {\it a priori} control $\|\phi\|_{L^r}\leq \|\phi\|_{H^{1,a}}$, $2\leq r<p^*$ and the Gagliardo-Nirenberg's, that one can extend them to $W^{1,r}, 1\leq r<p^*$. We formulate this in the following corollary. 
\begin{corollary}
	\label{cor:94} 
	Under the assumptions made in Proposition \ref{app:pr10}, we have for any integer $k$, 
	\begin{equation}
		\label{q:15} 
			\|\phi\|_{W^{1,r}(|x|\sim N)}\leq C_k N^{-k}, 1\leq r<p^*.
	\end{equation}
\end{corollary}
Next, we further expand the decay of the $L^2$ norm to any norm $L^q, 1\leq q\leq \infty$. Obviously, we already have that for $q<p^*$, so now we concentrate on the case $q>p^*$. 
\begin{proposition}
	\label{app:pr20} 
	Let $d\geq 1, a\in (0,1), p>1$ satisfy \eqref{527}. Let also $\phi$ be a positive   $C^\infty(|x|>1) \cap H^{1,a}(\rd)$ solution of \eqref{120}.
	Then, for each $k\geq 0$, and for each  integer $N\in \cn$, there exists $C_k$, so that  for all $|x|>1$, 
	\begin{equation}
		\label{1000} 
		\phi(x)\leq C_N |x|^{-N}. 
	\end{equation}
\end{proposition}
\begin{proof}
	We only consider the case $d\geq 3$, as the cases $d=1,2$ are already covered by  Sobolev embedding and Corollary \ref{cor:94}. Next, 
	we make  the observation that due to \eqref{527}, we have 
	$$
	(p-1) \f{d}{2}<\f{2 d(1-a)}{d-2(1-a)}<p^*.
	$$
	Thus, fix  $p_0\in \left(\max((p-1)\f{d}{2}, p), p^*\right)$. 
	
	Our starting point for this is an energy estimate for \eqref{900}. Specifically, divide the equation \eqref{900} by $|x|^{2a}$. Take a dot product of the resulting identity with $\phi_k^{q-1}$, with $q>2$ to be described shortly. After integration by parts, we have a very crude estimate\footnote{where we have gave up favorable powers of $2^k$ on the right-hand side etc.}  in the form 
	\begin{equation}
		\label{970} 
		\f{4(q-1)}{q^2} \int |\nabla \phi^{\f{q}{2}}|^2 \leq C \left(\int_{|x|\sim 2^k}  \phi^{p+q-1} dx+\int_{|x|\sim 2^k} \phi^q\right). 
	\end{equation}
	As long as $q: p+q-1=p_0<p^*$, we have a control by $2^{-Nk}$ for any $N$ on the right hand side, by Corollary \ref{cor:94}, whence 
	$
	\|\nabla \phi_k^{\f{q}{2}}\|_{L^2}<C_k 2^{-Nk}. 
	$
	By Sobolev embedding, we obtain 
	$$
	\|\phi_k\|_{L^{\f{qd}{d-2}}}\leq C_k 2^{-Nk}. 
	$$
	It follows that we are able to improve the integrability of $\phi_k$ (with estimates $C_k 2^{-Nk}$) from $L^{p_0}$ to $L^{p_1}: p_1:=\f{qd}{d-2}=\f{d}{d-2}(p_0-p+1)$. Note 
	$$
	p_1-p_0=\f{d}{d-2}(p_0-p+1)-p_0=\f{1}{d-2}(2 p_0 - d(p-1))>0,
	$$
	by the choice of $p_0$. 
	
	We can of course iterate this argument, as follows - each time we have the estimate $\|\phi_k\|_{L^{p_{n}}}\leq C_k 2^{-Nk}$, we can improve it to 
	$\|\phi_k\|_{L^{p_{n+1}}}\leq C_k 2^{-Nk}$, where 
	\begin{equation}
		\label{l:40}
		p_{n+1}=\f{d}{d-2}\left(p_{n}-p+1\right). 
	\end{equation}
	We can write a formula for the recursive relation \eqref{l:40} displayed herein, namely, 
	$$
	p_n=\f{(p-1) d}{2}+ \left(\f{d}{d-2}\right)^n (p_0-\f{(p-1) d}{2}). 
	$$
	By the choice of $p_0$ (note $p_0>\f{(p-1) d}{2}$),  $\lim_n p_n=+\infty$, whence $\|\phi_k\|_{L^r}\leq C_k 2^{-Nk}$ for all $1\leq r<\infty$. 
	
	Going back to \eqref{900}, we can now immediately obtain similar estimates for $\|\De \phi_k\|_{L^r}\leq C_k 2^{-Nk}$ for any $1\leq r<\infty$. Sobolev embedding then implies \eqref{1000}. 
\end{proof}

\subsection{$L^2$ exponential decay at infinity}
Now that we have established that the wave $\phi$ decays at infinity, see \eqref{1000}, we proceed to upgrade this result to exponential decay. The standard result in this direction is the classical Agmon's argument, which we mimick herein. The first step is such a decay, but in $L^2$ sense. 
\begin{proposition}
	\label{prop:18} 
	Let $\phi$ be a wave satisfying \eqref{120}. Then, there exist constants $C, \de$, so that 
	\begin{equation}
		\label{1010} 
		\int_{\rd} e^{\de |x|^{1-a}} \phi^2(x) dx\leq C. 
	\end{equation}
\end{proposition}
We provide the proof of the Proposition, in a few steps. 
We follow the approach in \cite{HS}. We begin with the following lemma, which applies to general second order divergence form operators, similar to  $\cl_\pm$. 
\subsubsection{First technical lemma}
	\begin{lemma}
		\label{v:1}
		Let $V:\rd\to \rone$, $V(x)>0$,  so that $\lim_{|x|\to \infty}=0$, and denote $H=-\nabla \cdot (|x|^{2a}  \nabla ) - V$.  Let $\tilde{V}(s):=\sup_{|x|=s} V(x)$ and consider the set $\cf:=\{x: \tilde{V} (x)<\f{\om}{2}\}$, and its compact complement 
		$\ca:=\cf^{c}=\{x: \tilde{V}(x)\geq \f{\om}{2}\}$. 
		
		Let $\de>0$ and $f:\rd\to \rone$ be a radial function, with 
		\begin{equation}
			\label{v:30}
				|f'(\rho)|\leq  (1-\de) \f{\sqrt{(\om-\tilde{V}(\rho))_+}}{\rho^a}. 
		\end{equation}
	Then, 	there is the estimate 
		\begin{equation}
			\label{v:20}
			\langle e^{f}\psi,(H+\om) e^{-f}\psi\rangle \geq  
			\f{\de \om}{2} \|\psi\|_{2}^{L^2}.
		\end{equation}
	\end{lemma}
{\bf Remark:} Clearly, we can take 
\begin{equation}
	\label{v:15} 
	f(\rho):=(1-\de)\int_1^\rho \frac{\sqrt{(\om-\tilde{V}(s))_+}}{s^a} ds,
\end{equation}
which will satisfy \eqref{v:30}.  Note that as $\lim_{\rho\to \infty} \tilde{V}(\rho)=0$, we have that $f(\rho)\sim \rho^{1-a}, \rho>>1$. 
	\begin{proof}
		We calculate 
		\begin{eqnarray*}
		& & 	\langle e^{f}\psi,-\nabla\cdot(|x|^{2a}\nabla)e^{-f}\psi \rangle  =\langle\nabla e^{f}\psi,(|x|^{2a}\nabla e^{-f}\psi\rangle\\
			 &=& \langle e^{f}\nabla \psi+e^{f}\psi \nabla f,|x|^{2a} e^{-f}\nabla \psi  -|x|^{2a}e^{-f}\psi  \nabla f\rangle\\
			&=& \langle\nabla \psi ,|x|^{2a}\nabla \psi \rangle-\langle\nabla \psi ,|x|^{2a}\psi  \nabla f\rangle+\langle \psi  \nabla f,|x|^{2a}\nabla \psi  \rangle-\langle\psi  \nabla f,|x|^{2a}\psi  \nabla f\rangle=\\
			&=& \langle\nabla \psi,|x|^{2a}\nabla \psi \rangle - \int_{\rd} |x|^{2a} |\nabla f|^2 \psi^2.
		\end{eqnarray*}
	Combining everything, 
	\begin{align*}
		\langle e^{f} \psi ,(H+\om) e^{-f}\psi \rangle & =\langle\nabla \psi ,|x|^{2a}\nabla \psi \rangle-  \int_{\rd} |x|^{2a} |\nabla f|^2 \psi^2  -\langle \psi,V\psi \rangle+\om \|\psi\|^2 \\
		& \geq\langle \psi,(\om-V -|x|^{2a}|\nabla f|^{2})\psi\rangle.
	\end{align*}
We now estimate the last expression from below. Since $supp\  \psi\subset \cf$, we have that for $x\in \cf$, $\rho=|x|$, 
\begin{eqnarray*}
	\om-V(x) -|x|^{2a}|\nabla f|^{2}\geq 
	\om-\tilde{V}(\rho) -\rho^{2a}|f'(\rho)|^{2} \geq \de(\om-\tilde{V}(\rho))\geq \f{\de\om}{2}.
\end{eqnarray*}
Thus, 
$$
\langle e^{f} \psi ,(H+\om) e^{-f}\psi \rangle \geq \f{\de\om}{2} \|\psi\|^2_{L^2},
$$
as claimed. 
	\end{proof}
	We now proceed to the next technical  lemma. 
	\subsubsection{Second  technical lemma}
	\begin{lemma}
		\label{The-second-one.} 
		Suppose $(H+\om) \phi=0$ and $f$ is a bounded function. 
		 Let $\eta$ be a smooth
		bounded function with $\text{supp}\ |\nabla\eta|$ compact. Set $\psi:=\eta e^{f}\phi$. Then, 
		\begin{equation}
			\label{v:50}
			\langle e^{f}\psi,(H+\om) e^{-f}\psi\rangle=\langle\xi e^{2f}\phi,\phi\rangle,
		\end{equation}
		where $\xi=|\nabla\eta|^{2}+2 \eta (\nabla\eta\cdot\nabla f).$
	\end{lemma}
\begin{proof}
	Observe that since  function $f$ is bounded 
	$e^{f}\psi\in L^{2}(\rd)$. We now compute
	$$
	\langle e^{f}\psi,(H+\om)e^{-f} \psi\rangle=\langle\eta e^{2f}\phi,(H+\om) \eta\phi\rangle. 
	$$
	Using the fact that $(H+\om)\phi=0$,
	we get that 
	\begin{equation}
		\label{eq:lem2_operator_equivalence}
		\langle\eta e^{2f }\phi,(H+\om)\eta\phi\rangle=\langle\eta e^{2f}\phi,[-\Delta\eta-2\nabla\eta\cdot\nabla]\phi\rangle.
	\end{equation}
Integration by parts yields 
$$
	\langle\eta e^{2f}\phi,[-\Delta\eta-2\nabla\eta\cdot\nabla]\phi\rangle = \int_{\rd} \left[ |\nabla \eta|^2 +2 \eta \nabla \eta\cdot \nabla f \right] e^{2 f} \phi^2 
$$
whence \eqref{v:50} follows. 
 
\end{proof}
	\subsubsection{Proof of Proposition \ref{prop:18}}
	We are now ready for the proof of the $L^2$ exponential bounds. Define now $f$ as in \eqref{v:15}, and also for every $\al>0$, introduce $f_\al(x):=\f{f(x)}{1+\al f(x)}$. Consequently, consider the sets 
	$$
	\cf=\{x: \tilde{V}(x)>\f{\om}{2}\}, \ca=\{x: \tilde{V}(x)<\f{\om}{4}\}.
	$$
	Take $\eta: \eta(x)=1$ on $\cf$, while $\eta(x)=1, x\in \ca$. Note that $supp\ \nabla \eta\subset \{x: \tilde{V}(x)\in [\f{\om}{4}, \f{\om}{2}]\}$, which is a compact subset of $\rd$. 
	
	Define $\psi:=\eta e^{f_{\alpha}}\phi$. By the calculation in Lemma \ref{v:1},  we have 
	\begin{align*}
		\langle e^{f_{\alpha}}\psi,(H+\om)e^{-f_{\alpha}}\psi\rangle & \geq\langle\psi,(\om - V(x) -|x|^{2a}|\nabla f_{\alpha}|^{2})\psi\rangle\\
		& \geq\langle\psi,(\om- V-|x|^{2a}|\nabla f|^{2})\psi\rangle\geq\f{\de \om}{4} \|\psi\|^{2},
	\end{align*}
	as by direct verification $|\nabla f_{\alpha}|^{2}\leq|\nabla f|^{2}$. 
	Therefore by lemma
	(\ref{The-second-one.}), we get 
	\begin{align}
	\f{\de \om}{4} \|\psi\|^{2}_{L^2} & \leq\langle e^{f_{\alpha}}\psi,(H+\om)e^{-f_{\alpha}}\psi\rangle\nonumber \\
		& \leq|\langle\xi e^{2f_{\alpha}}\phi,\phi\rangle|\leq \sup_{x\in\text{supp}\nabla\eta}|\xi e^{2f_{\alpha}}| \|\phi\|^{2}_{L^2}\label{eq:bound independent of alpha}
	\end{align}
	Since $e^{2f_{\alpha}}\leq e^{2f}$ and the support of $\nabla\eta$ is
	compact, $e^{2f}$ is bounded in its support and hence we can take
	$\alpha\to 0+$  in \eqref{eq:bound independent of alpha}. 
	 thereby
	obtaining a bound independent of $\alpha$ on the right side of equation
	(\ref{eq:bound independent of alpha}). Thus, 
	\begin{equation}
		\label{v:70} 
		\|\eta e^{f}\phi\|_{L^2}\leq C \|\phi\|_{L^2}.
	\end{equation}
	So, from \eqref{v:70}, 
	$$
	\int_{\eta(x)=1} e^{2 f(x)} \phi^2\leq 	\|\eta e^{f}\phi\|_{L^2}^2\leq  C \|\phi\|^2_{L^2}.
	$$
Recall that $f(|x|)\sim |x|^{1-a}$, so that the last estimate implies \eqref{1010}. 

	 \subsection{$L^\infty$ exponential decay at infinity }
	 We start with the following observation, which almost resolves the issue. By  interpolating between the exponential $L^2$ decay estimate \eqref{1010} and either \eqref{q:15} (for $r=1$) or \eqref{1000} (for $r=\infty$), we obtain the exponential bound 
	 \begin{equation}
	 	\label{q:20} 
	 	\|\phi\|_{L^\ga(B)}\leq C_\ga e^{-\de_\ga |x|^{1-a}}, 1\leq \ga<\infty.
	 \end{equation}
 
	 From Proposition \ref{app:pr20}, specifically \eqref{1000}, we know that $\phi$ has polynomial decay away from the origin, in
	 particular $\phi(x)\leq C_{N}|x|^{-N}$. So, consider the region $|x|>R,$
	 where $R$ is chosen such that $\|\phi\|_{L^{\infty}(|x|>R)}<1$. 
	 
	 For $x\in B_{1}(x_{0})$ (ball of radius 1 around $x_{0}$), we introduce the cutoff function $\chi_B$ and $\phi_B:=\phi \chi_B$.  As a result and due to the positivity of the Green's function for $(\omega-\Delta)^{-1}$, as in \eqref{890} we can estimate 
	 \begin{equation}
	 	\label{q:5} 
	 	 \phi_{B}(x)\leq\left(\omega-\Delta\right)^{-1}\left[\frac{\phi^{p}\chi_{B}}{|x|^{2a}}+2a\sum_{j}\frac{x_{j}}{|x|^{2}}(\partial_{j}\phi)\chi_{B}-2\sum_{j}\partial_{j}\phi\partial_{j}\chi_{B}-\phi\Delta\chi_{B}\right],
	 \end{equation}
	 where $\omega:=\min_{x\in B_{2}(x_{0})}\{\frac{1}{|x|^{2a}}\}\sim |x|^{-2 a}$.

	 We now estimate $L^{\infty}$ norm of each term on the right hand
	 side of \eqref{q:5}. We have by Young's inequality, for some $1<p<\frac{d}{d-2}$ and $\frac{1}{p}+\frac{1}{q}=1$, 
	 \begin{eqnarray*}
	 	\|\omega{}^{\frac{d}{2}-1}
	 	Q(\sqrt{w}\cdot)*\frac{\phi^{p}\chi_{B}}{|x|^{2a}}\|_{L^{\infty}} & \leq &  \|\omega{}^{\frac{d}{2}-1}Q(\sqrt{w}x)\|_{p}\|\phi^{p}\chi_{B}\|_{q}
	 	 \leq\omega^{\frac{d}{2}-1}\omega^{-\frac{d}{2p}}\|Q\|_{p}\|\phi^{p}\chi_{B}\|_{q} \\
	 	 &\leq & \omega^{\frac{d}{2}-1}\omega^{-\frac{d}{2p}}\|Q\|_{p}
	 	 \|\phi\|_{L^{p q}(B)}^{p}  \leq C|x|^{c(a,p,d)} e^{-\delta |x_0|^{1-a}}\leq C e^{-\delta_1 |x_0|^{1-a}}
	 \end{eqnarray*}
 where we have used that $\om\sim |x|^{-2a}$, and later, we have absorbed powers of $|x|$ into the exponential term, by taking smaller $\delta_1$. 
 
 Similarly, for the terms $\omega{}^{\frac{d}{2}-1}Q(\sqrt{w}\cdot)*\frac{x_{j}}{|x|^{2}}(\partial_{j}\phi)\chi_{B}\text{ and }\omega{}^{\frac{d}{2}-1}Q(\sqrt{w}\cdot)*\partial_{j}\phi\partial_{j}\chi_{B}$
 we use integration by parts and the fact that $\partial_{j}Q\in L^{p}$, 
 - where $1<p<\frac{d}{d-1}, \f{1}{p}+\f{1}{q}=1$, along with the fact that $\partial_{j}\chi_{B}$
 is bounded to obtain bounds similar as above, hence we get,
 $$
 \|\omega{}^{\frac{d}{2}-1}Q(\sqrt{w}\cdot)*\left(2a\sum_{j}\frac{x_{j}}{|x|^{2}}(\partial_{j}\phi)\chi_{B}+2\sum_{j}\partial_{j}\phi\partial_{j}\chi_{B}\right)
 \|_{L^{\infty}}\leq C_{\omega}\|\phi\|_{L^{q}(B)}\leq C e^{-\delta_{2}|x_0|^{1-a}}.
 $$
 For the final term, by a similar estimate we obtain the bounds, 
 $$
 \|\omega{}^{\frac{d}{2}-1}Q(\sqrt{w}\cdot)*\phi\Delta\chi_{B}\|_{L^{\infty}(B_{1}(x_{0}))}\leq C\|\phi\|_{L^{q}(B)}\leq C e^{-\delta_{3}|x_0|^{1-a}}.
 $$
 Combining the three estimates and we obtain a exponential decay bound
 for $\phi$, 
 $$
 \phi(x)\leq C e^{-\delta|x_0|^{1-a}}\ 
 $$
 where $\delta=\min\{\delta_{1},\delta_{2},\delta_{3}\}.$ 
 
  {\bf Data availability statement:} We affirm that our paper has no associated data.

 \appendix
 
 \section{Proof of Proposition \ref{prop:Poh}} 
 
 Consider a smooth cut off function $\chi(x)$ supported
 on $|x|<2$, with $\chi(x)=1$ for $|x|<1$ and $\| \nabla \chi \| _{L^{\infty}}<C $, and define $\chi_{\lambda}(x):=\chi(\frac{x}{\lambda})$.
Define the function $\phi_{\epsilon,N}=\phi(1-\chi_{\epsilon})\chi_{N}.$
Note that $\lim_{\epsilon\to0}(1-\chi_{\epsilon})=1$ and $\lim_{N\to\infty}\chi_{N}=1$.
Note that $\phi_{\epsilon,N} $ is supported away
from the origin, $|x|>\epsilon$ and away from infinity.  

Now $\phi_{\epsilon,N}$ is smooth and has the exact decay properties
as $\phi$. Multiplying equation \eqref{120} by $\phi_{\epsilon,N}$
and integrating by parts we obtain
\begin{align*}
&\int\frac{1}{N}|x|^{2a}\nabla  \phi\phi(1-\chi_{\epsilon})(\nabla\chi)_{N}\ dx+\int\frac{1}{\epsilon}|x|^{2a}\nabla\phi\phi(\nabla\chi)_{\epsilon}\chi_{N}\ dx\\
+&\int|x|^{2a}|\nabla\phi|^{2}(1-\chi_{\epsilon})(\nabla\chi_{N})\ dx+\omega\int|\phi|^{2}(1-\chi_{\epsilon})(\nabla\chi_{N})\ dx-\int|\phi|^{p+1}(1-\chi_{\epsilon})(\nabla\chi_{N})\ dx  =0
\end{align*}
The problematic term are $\int\frac{1}{N}|x|^{2a}\nabla\phi\phi(1-\chi_{\epsilon})(\nabla\chi_{N})\ dx\text{ and }\int\frac{1}{\epsilon}|x|^{2a}\nabla\phi\phi(\nabla\chi_{\epsilon})\ dx$,
which are bounded above,
\begin{eqnarray*}
\int\frac{1}{N}|x|^{2a}\nabla\phi\phi(1-\chi_{\epsilon})(\nabla\chi_{N})\ dx & \leq & \frac{C}{N}\left(\int|x|^{2a}|\nabla\phi|^{2}\ dx\right)^{1/2}\left(\int_{|x|\sim N}|x|^{2a}|\phi|^{2}\ dx\right)^{1/2}\\
 & \leq & CN^{a-1}\left(\int|x|^{2a}|\nabla\phi|^{2}\ dx\right)^{1/2}\left(\int_{|x|\sim N}|\phi|^{2}\ dx\right)^{1/2}.
\end{eqnarray*}
The right hand side is bounded as $\phi\in H^{1,a}$ and the integral
on the left vanishes as $N\to\infty$ as $a<1$. For the other problematic
term we have, 
\begin{eqnarray*}
\int\frac{1}{\epsilon}|x|^{2a}\nabla\phi\phi(\nabla\chi_{\epsilon})\ dx & \leq & \frac{C}{\epsilon}\left(\int|x|^{2a}|\nabla\phi|^{2}\ dx\right)^{1/2}\left(\int_{|x|\sim\epsilon}|x|^{2a}|\phi|^{2}\ dx\right)^{1/2}\\
 & \leq & \frac{C}{\epsilon}\left(\int|x|^{2a}|\nabla\phi|^{2}\ dx\right)^{1/2}\left(\int_{|x|\sim\epsilon}|\phi|^{2q}\ dx\right)^{\frac{1}{2q}}\left(\int_{|x|\sim\epsilon}|x|^{2a\frac{q}{q-1}}\right)^{\frac{q-1}{2q}}\\
 & \leq & C\epsilon^{a+d\frac{q-1}{2q}-1}\left(\int|x|^{2a}|\nabla\phi|^{2}\ dx\right)^{1/2}\left(\int_{|x|\sim\epsilon}|\phi|^{2q}\ dx\right)^{\frac{1}{2q}}
\end{eqnarray*}

We see that $a+d\frac{q-1}{2q}=1$, exactly when $2q=\frac{2d}{d-2(1-a)},$which
is precisely the end point bound in the Caffarelli-Kohn-Nirenberg
inequality \eqref{25}. Thus the integrals are bounded above and by
dominated convergence theorem we may pass limits inside
the integral to obtain, 
\begin{equation}
\label{1400}
\int|x|^{2a}|\nabla\phi|^{2}\ dx+\omega\int|\phi|^{2}\ dx-\int|\phi|^{p+1}\ dx=0.
\end{equation}

Now, we multiply equation\eqref{25} by $x\cdot(\nabla\phi)(1-\chi_{\epsilon})\chi_{N}$
and integrate by parts term wise. For the first term we obtain,
$$
\int(-\nabla\cdot(|x|^{2a}\nabla\phi))(x\cdot(\nabla\phi)(1-\chi_{\epsilon})\chi_{N})=\int(|x|^{2a}\nabla\phi)\cdot(\nabla(x\cdot(\nabla\phi)(1-\chi_{\epsilon})\chi_{N}))\ dx
$$
and 
\begin{align*}
& \int(|x|^{2a}\nabla\phi)\cdot(\nabla(x\cdot(\nabla\phi)(1-\chi_{\epsilon})\chi_{N}))\ dx  =\int|x|^{2a}|\nabla\phi|^{2}(1-\chi_{\epsilon})\chi_{N}\\
 & \ \ \ +\sum_{i}\sum_{j}\int|x|^{2a}\phi_{i}x_{j}(\phi_{ji})(1-\chi_{\epsilon})\chi_{N}+\text{ derivatives on cutoff}.
\end{align*}
A subsequent integration by parts on $\sum_{i,j}\int|x|^{2a}\phi_{i}x_{j}(\phi_{ji})(1-\chi_{\epsilon})\chi_{N}$
yields, 
\begin{align*}
\int-\nabla\cdot(|x|^{2a}\nabla\phi)(x\cdot(\nabla\phi)(1-\chi_{\epsilon})\chi_{N}) & =\frac{-d-2a+2}{2}\int|x|^{2a}|\nabla\phi|^{2}(1-\chi_{\epsilon})\chi_{N}\\
 & \ \ +\text{ terms with derivative on cutoff ,}
\end{align*}
 where the terms with derivatives on cutoff are of the form, $\int|x|^{2a}\phi_{i}\phi_{j}\frac{x_{j}}{\epsilon}(\nabla\chi)_{\epsilon}\chi_{N}\ dx$
and\newline $\int|x|^{2a}\phi_{i}\phi_{j}\frac{x_{j}}{N}(1-\chi_{\epsilon})(\nabla\chi)_{N}\ dx$
for some $1\leq i,j\leq d.$ Each of the terms are bounded above by
$C\int_{|x|\sim\epsilon}|x|^{2a}|\nabla\phi|^{2}\ dx$ and $C\int_{|x|\sim N}|x|^{2a}|\nabla\phi|^{2}\ dx$
respectively and $\lim_{\epsilon\to0}\int_{|x|\sim\epsilon}|x|^{2a}|\nabla\phi|^{2}\ dx=0$
 and $\lim_{N\to\infty}\int_{|x|\sim N}|x|^{2a}|\nabla\phi|^{2}\ dx=0$
as $\phi\in H^{1,a}$. 

For the second term, integrating by parts, we have
\begin{align*}
\int\omega\phi(x\cdot(\nabla\phi)(1-\chi_{\epsilon})\chi_{N})\ dx & =-\sum_{j}\int\omega\phi^{2}(1-\chi_{\epsilon})\chi_{N}-\sum_{j}\omega\phi_{j}x_{j}\phi(1-\chi_{\epsilon})\chi_{N}\\
 & \ \ +\text{ terms with derivative of cutoff. }
\end{align*}
Thus, 
$$
\int\omega\phi(x\cdot(\nabla\phi)(1-\chi_{\epsilon})\chi_{N})\ dx=-\frac{d}{2}\omega\int\phi^{2}(1-\chi_{\epsilon})\chi_{N}+\text{ terms with derivative of cutoff. }
$$
The terms with derivative of the cutoff are or the form, 
$$ 
\int_{|x|\sim\epsilon}\omega\phi^{2}\frac{x_{j}}{\epsilon}(\nabla\chi)_{\epsilon}\chi_{N}, \int_{|x|\sim N}\omega\phi^{2}\frac{x_{j}}{N}(1-\chi{}_{\epsilon})(\nabla\chi)_{N}.
$$
They  vanish
when we pass the limits, since $\phi\in L^2$. 

For the last term, we have 
\begin{align*}
& -\int\phi^{p}(x\cdot(\nabla\phi)(1-\chi_{\epsilon})\chi_{N})  =\sum_{j}\int\phi(\phi^{p}x_{j}(1-\chi_{\epsilon})\chi_{N})_{j}\\
 & =d\int\phi^{p+1}(1-\chi_{\epsilon})\chi_{N}+\int p\phi^{p}x_{j}\phi_{j}(1-\chi_{\epsilon})\chi_{N}+\text{derivatives of cutoff}.
\end{align*}

Thus 
$$
-\int\phi^{p}(x\cdot(\nabla\phi)(1-\chi_{\epsilon})\chi_{N})=\frac{d}{p+1}\int\phi^{p+1}(1-\chi_{\epsilon})\chi_{N}\ dx+\text{ derivatives on cutoff .}
$$

Here the terms with derivative on cutoff are of the form, $\int_{|x|\sim\epsilon}\phi^{p+1}\frac{x_{j}}{\epsilon}(\nabla\chi)_{\epsilon}\chi_{N}\ dx$
and \newline $\int_{|x|\sim N}\phi^{p+1}\frac{x_{j}}{N}(1-\chi{}_{\epsilon}(\nabla\chi)_{N}\ dx$.
These terms are bounded above by $\int|\phi|^{p+1}$ and will vanish
when passing the limits inside. \\
By combining all the above terms and passing limits inside the integral,
we obtain the identity, 
\begin{equation}
\label{1450}
\frac{d-2+2a}{2}\int|x|^{2a}|\nabla\phi|^{2}\ dx+\frac{d}{2}\omega\int\phi^{2}\ dx=\frac{d}{p+1}\int\phi^{p+1}\ dx
\end{equation}
 The proposition is proved by combining equations \eqref{1400} and \eqref{1450}.

 \section{Proof of Proposition \ref{prop:10}} 

 	We start with a bounded in $H^{1,a}$ set $\cf$, and we verify the Riesz-Kolmogorov criteria. From the Caffarelli-Kohn-Nirenberg's and Young's inequality, we infer that 
 	$
 	\|f\|_{L^q(\rd)}\leq C \|f\|_{H^{1,a}},
 	$
 	whence the first condition in Riesz-Kolmogorov is settled. 
 	
 	Next, we verify the second condition. To this end, 
 	let $\psi\in C^{\infty}(\rd)$ be such that $0\leq\psi\leq1$ and $\psi(x)=\begin{cases}
 		1, & |x|>1\\
 		0, & |x|\leq1/2
 	\end{cases}$. For $R>0$, let $\psi_{R}(x):=\psi(\frac{x}{R})$. Note that if
 	$q$ satisfies the condition \eqref{25}, then
 	it also satisfies the condition $\f{1}{q}=\f{1}{2}-\f{\theta_0}{d}$  for
 	the Gagliardo-Nirenberg inequality, for some $\theta=\theta_{0}\in(0,1)$.
 	By properties of $\psi_{R}$ and the Cauchy-Schwartz inequality, we
 	get
 	\begin{align*}
 		\int|\nabla(u\psi_{R})|^{2}\ dx\leq & \ 2\int|(\nabla u)\psi_{R}|^{2}\ dx+2\int|u(\nabla\psi_{R})|^{2}\ dx\\
 		\leq & \ \frac{2^{2a+1}}{R^{2a}}\int_{|x|>R/2}|x|^{2a}|\nabla u|^{2}\ dx+\frac{2\|\nabla\psi\|_{L^{\infty}}}{R^{2}}\int_{|x|>R/2}|u|^{2}\ dx
 	\end{align*}
 	
 	Now choose $R>0$ such that $\max\left\{ \frac{2^{2a+1}}{R^{2a}},\frac{2\|\nabla\psi\|_{L^{\infty}}}{R^{2}}\right\} <\epsilon.$
 	Then by the Gagliardo-Nirenberg inequality we have,
 	\begin{align*}
 		&	\int_{|x|>R}|u|^{q}\ dx\leq  \int_{\mathbb {R}^{d}}|u\psi_{R}|^{q}\ dx
 		\leq  \ C\left(\int|\nabla(u\psi_{R})|^{2}\ dx\right)^{\frac{\theta_{0}}{2}q}\left(\int|u\psi_{R}|^{2}\ dx\right)^{\frac{(1-\theta_{0})}{2}q}\\
 		\leq & \ C\epsilon^{\theta_{0}\frac{q}{2}}\left(\int_{|x|>R/2}|x|^{2a}|\nabla u|^{2}\ dx+\int_{|x|>R/2}|u|^{2}\ dx\right)^{\frac{\theta_{0}}{2}q}\left(\int_{|x|>R/2}|u|^{2}\ dx\right)^{(1-\theta_{0})q/2}\\
 		\leq & \ C\epsilon^{\theta_{0}\frac{q}{2}}\|u\|_{H^{1,a}}^{q}.
 	\end{align*}
 	Clearly then, 
 	\begin{equation}
 		\label{1201} 
 		\lim_{R\to \infty} \sup_{u\in \cf} \int_{|x|>R}|u|^{q}\ dx =0.
 	\end{equation}
 	which is the second requirement in Riesz-Kolmogorov. 
 	
 	Finally, we check the third condition there. Since, for say $|y|<1, R>1$, we have 
 	\begin{eqnarray*}
 		\int_{|x|>R} |u(x+y)-u(x)|^q dx \leq C_q \int_{|x|>R/2} |u(x)|^q dx,
 	\end{eqnarray*}
 	and in view of \eqref{1201}, it suffices to control 
 	$$
 	\sup_{u\in \cf} \int_{|x|<R} |u(x+y)-u(x)|^q dx 
 	$$
 	Observer that by \ref{25}
 	we have $q=q(\theta,a)$, so for a fixed $a$, by chosing an appropriate
 	$\theta$, we also get that $u \in L^{\tilde{q}}$ for some $\tilde{q}>q.$
 	Specifically, from the log convexity of the $L^{p}$
 	norms to conclude that, 
 	\begin{equation}
 		\label{1405}
 		\|\tau_{y}u-u\|_{L^{q}(B_{R})}\leq\|\tau_{y}u-u\|_{L^{\tilde{q}}(B_{R})}^{\delta}\|\tau_{y}u-u\|_{L^{p}(B_{R})}^{(1-\delta)}\leq
 		2\|u\|_{L^{\tilde{q}}}^{\delta}\|\tau_{y}u-u\|_{L^{p}(B_{R})}^{(1-\delta)}, 
 	\end{equation}
 	for some $\delta\in(0,1)$and $p>1$.  We shall now focus our attention
 	on $\|\tau_{y}u -u \|_{L^{p}(B_{R})}.$ 
 	
 	Note that by Jensen's, we have the estimate, 
 	$$
 	\int_{|x|<R}|u(x+y)-u(x)|^{p}\ dx\leq C\int_{|x|<2R}|u_{R}|^{p}\ dx,
 	$$
 	where $u_{R}(x)=u(x)\eta(x/2R),$, $\eta$ is a cut-off function defined below, while  by mean-value theorem and Fubini's theorem we get
 	the inequality, 
 	$$
 	\int_{|x|<R}|u(x+y)-u(x)|^{p}\ dx\leq|y|^{p}\int_{|x|<2R}|\nabla u_{R}|^{p}\ dx.
 	$$
 	By  the interpolation identity $\dot{W}^{s,p}=[L^{p},\dot{W}^{1,p}]_{s}$, which
 	is the homogenous fractional Sobolev space, we have the interpolation
 	estimate, for all $0<s<1$, 
 	$$
 	\left(\int_{|x|<R}|u(x+y)-u(x)|^{p}\ dx\right)^{1/p}\leq C|y|^{s}\|D^{s}u_{R}\|_{L^{p}}
 	$$
 	Clearly, as $|y|$ can be chose to be as small as we need, the third condition in the RK criteria follows from the estimate 
 	\begin{equation}
 		\label{130} 
 		\|D^{s}u_{R}\|_{L^{p}}\leq C_R \|u_R\|_{H^{1,a}},
 	\end{equation}
 	which will hold for $a\in (0,1), 0<s<\f{d}{7}$, $1<p<\f{2d}{d+7 s}$.  By elementary rescaling in \eqref{130}, we can reduce it to the case\footnote{at the expense of the dependence of $C_R$ in \eqref{130} on $R$}  $R=1$. We provide the proof of this estimate next.

 \section{Proof of the embedding 
 	$W^{s,p}\hookrightarrow H^{1,a}$} 
 \begin{lemma}
 	\label{le:app} 
 	Let $d\geq 1$ and $a\in (0,1), 0<s<\f{d}{7}$. Then, there exists $\de>0$, so that for all $p: 1<p<\f{2d}{d+7 s}$, there is the embedding estimate 
 	\begin{equation}
 		\label{135} 
 		\|D^{s}u\|_{L^{p}(\rd)}\leq C \|u\|_{H^{1,a}(\rd)},
 	\end{equation}
 	for all functions $u: supp u\subset B(1)$. 
 \end{lemma}
 {\bf Remark:} The restrictions $0<s<\f{d}{7}, 1<p<\f{2d}{d+7 s}$ are  manifestly not optimal and it may be improved. The point here is that for small enough $s>0$ and for $p>1$, which is sufficiently close to $1$, the embedding  \eqref{135} holds true.  
 \begin{proof}
 	We introduce a function $\eta\in C_0^\infty: 0\leq \eta\leq 1$, so that  $\eta(\xi)=1, |\xi|<1, \eta(\xi)=0, |\xi|>2$. Next, a function $\vp: \vp^2(\xi):=\eta(\xi)-\eta(2\xi); \vp_j(\xi)=\vp(2^j \xi)$. So that $\eta=\sum_{j=0}^\infty \vp^2_j(\xi)$. 
 	Clearly, 
 	$
 	u=u\eta=\sum_{j=0}^\infty u\vp_j^2.
 	$
 	We have, as a consequence of the product estimate \eqref{KP}, 
 	\begin{eqnarray*}
 		\|D^{s}u\|_{L^{p}(\rd)}\leq \sum_{j=0}^\infty 
 		\|D^{s}[\vp_j u_j]\|_{L^{p}}\leq \sum_{j=0}^\infty  
 		(\|D^s \vp_j\|_{L^{\f{2p}{2-p}}}  \|u_j\|_{L^2}+ 
 		\|D^s u_j\|_{L^2}  \|\vp_j\|_{L^{\f{2p}{2-p}}}),
 	\end{eqnarray*}
 	where we have used the notation $u_j:=u\vp_j$. Computing 
 	$$
 	\|\vp_j\|_{L^{\f{2p}{2-p}}} \sim 2^{-jd\f{2-p}{2p}},  \|D^s \vp_j\|_{L^{\f{2p}{2-p}}} \sim 2^{j(s-d\f{2-p}{2p})}, 
 	$$
 	yields the bound 
 	$$
 	\|D^{s}u\|_{L^{p}(\rd)}\leq \sum_{j=0}^\infty ( 2^{j(s-d\f{2-p}{2p})} \|u_j\|_{L^2} +2^{-jd\f{2-p}{2p}} \|D^s u_j\|_{L^2}  )
 	$$
 	Since $s<d\f{2-p}{2p}$ by assumption, the first term in the sum is controlled by $\|u\|_{L^2}$. For the second term, we estimate by Young's 
 	\begin{equation}
 		\label{1410} 
 		2^{-jc_{p}}|\xi|^{2s}= \ 2^{(5js+2jas)}2^{-5js-2jas}2^{-jc_{p}}|\xi|^{2s}\\
 		\leq  \ s\left(2^{-5j-2ja}|\xi|^{2}\right)+(1-s) 
 		2^{-j\f{c_{p}-2as-5s}{(1-s)}}
 	\end{equation}
 	where $c_{p}=d\f{2-p}{p}$. 
 	
 	We can then estimate, as a consequence of \eqref{1410}, 
 	\begin{eqnarray*}
 		& & 	 \sum_{j=0}^\infty  2^{-jd\f{2-p}{2p}} \|D^s u_j\|_{L^2}  \leq  \sum_{j=0}^\infty  \left(2^{-jd\f{2-p}{p}}\int |\xi|^{2s} |\hat{u}_j(\xi)|^2 d\xi\right)^{\f{1}{2}} \leq \\
 		&\leq & 
 		C_s  \sum_{j=0}^\infty  \left( 2^{-2 a j} 2^{-5 j} \int |\xi|^{2} |\hat{u}_j(\xi)|^2 d\xi+ 2^{-j\f{c_p-s(2a+5)}{1-s}} \|u_j\|_{L^2}^2 \right)^{\f{1}{2}} \leq \\
 		&\leq & C_s (\|u\|_{L^2}+ 
 		\sum_{j=0}^\infty  2^{-\f{\de}{2}  j} 2^{- a j} \|\nabla u_j\|_{L^2})\leq C_s(\|u\|_{L^2}+  \sum_{j=0}^\infty  2^{-\f{5}{2}  j} \| |\cdot|^a \nabla u_j\|_{L^2}).
 	\end{eqnarray*}
 	The pointwise inequality, 
 	$$
 	|\nabla u_{j}|^{2}\leq C ( |\nabla u|^{2}|\varphi_{j}|^{2}+2^{2j}|u|^{2}\chi_{|x|\sim2^{-j}})
 	$$
 	allows us to control the last term as follows 
 	\begin{eqnarray*}
 		\sum_{j=0}^\infty  2^{-\f{5}{2}  j} \| |\cdot|^a \nabla u_j\|_{L^2}\leq C \sum_{j=0}^\infty  2^{-\f{5}{2}  j}
 		(\| |\cdot|^a \nabla u\|_{L^2}+ 2^j \|u\|_{L^2}\leq C \|u\|_{H^{1,a}}.
 	\end{eqnarray*}
 	Lemma \ref{le:app} is proved in full.

 \end{proof}

\section{Continuity at zero for radial solutions} 
Recall 
$$
1<p<\frac{d+2(1-a)}{d-2(1-a)}=1+\frac{4(1-a)}{d-2(1-a)}, p^{*}=\frac{2d}{d-2(1-a)}. 
$$
\begin{lemma}
	\label{le:app10} 
	Suppose $a\in (0,1)$ is as in \eqref{527}, and $\phi\in C^\infty(\rd\setminus\{0\})$ is a  solution of \eqref{120}, away from $0$, which is positive and radial.  Assume in addition that 
	$\int_{\rd} |x|^{2a} |\nabla \phi|^2 dx<\infty$. Then, $\phi$ is continuous at zero. In addition, 
	$$
	\begin{cases}
		\phi'(\rho) = - \f{\phi^p(0)-\om \phi(0)}{d} \rho^{1-2a} +o(\rho^{1-2a})\\
		\phi''(\rho) =  \f{2a-1}{d} (\phi^p(0)-\om \phi(0)) \rho^{-2a} + o(\rho^{-2a}).
	\end{cases}
	$$
\end{lemma}
\begin{proof}
	The case $d-1+2a<1$, which is $d=1, a<\f{1}{2}$,  follows from the Cauchy-Schwartz inequality, since 
	$$
	\int_0^1 |\phi'(\rho)| d\rho \leq \left(\int_0^1 \rho^{2a} |\phi'(\rho)|^2 d\rho \right)^{\f{1}{2}}\leq \left(\int  |x|^{2a} |\phi'(x)|^2 d\rho \right)^{\f{1}{2}}<\|\phi\|_{H^{1,a}},
	$$
	whence $\phi(\rho)=\phi(1)-\int_\rho^1 \phi'(\rho)$ has a limit as $\rho\to 0+$. 
	
	Assume henceforth, $d-1+2a\geq 1$. By the  Caffarelli-Kohn-Nirenberg's inequality, \eqref{30}, we have that 
	$$
	\|\phi\|_{L^r}\leq C_r \|\phi\|_{H^{1,a}}, \ \ 2\leq r<p^*.  
	$$
	We rewrite \eqref{120} as an ODE in the radial variable. After some algebraic manipulations, 
	\begin{equation}
		\label{800} 
		-(\rho^{d+2a-1} \phi'(\rho))' = \rho^{d-1}\left(\phi^p(\rho)-\om \phi(\rho) \right),  \rho>0
	\end{equation}
	Integrating in $(\rho_1, \rho_2): 0<\rho_1<\rho_2<1$ yields 
	$$
	\rho_1^{d+2a-1} \phi'(\rho_1) - \rho_2^{d+2a-1} \phi'(\rho_2)=\int_{\rho_1}^{\rho_2} \rho^{d-1} \left(\phi^p(\rho)-\om \phi(\rho) \right) d\rho.
	$$
	By the integrability $\int_{|x|<1}  \phi^p dx<\infty, \int_{|x|<1} \phi(x) <\infty$, we conclude that  $\rho_k^{d+2a-1} \phi'(\rho_k)$ is Cauchy, whenever $\rho_k\to 0+$. Thus, $A:=\lim_{\rho\to 0} \rho^{d+2a-1} \phi'(\rho) $ exists. Assuming for a contradiction, that $A\neq 0$, we conclude 
	$$
	\infty> \int |x|^{2a} |\nabla \phi|^2 dx\geq \f{A^2}{4} \int_0^{\rho_0} \rho^{d-1+2a} \f{1}{\rho^{2(d-1+2a)}} d\rho=+\infty
	$$
	if $d-1+2a\geq 1$. This is a contradiction, whence $A=0$. We can therefore write 
	\begin{equation}
		\label{810} 
		-\phi'(\rho) \rho^{d-1+2a}=\int_0^\rho (\phi^p(s) -\om \phi(s)) s^{d-1} ds. 
	\end{equation}
	It is now clear that for some small $\rho_0>0$, we have that $\phi'(\rho)<0, 0<\rho<\rho_0$, that is the function is decreasing. This follows from \eqref{810}, as we claim that we should have $\phi^p(s) -\om \phi(s)\geq 0, s\in (0, \rho_0)$. Otherwise, if $\phi^p(s) -\om \phi(s)<0$ in a neighborhood of $0$, one concludes from \eqref{810} that $\phi'>0$, whence $\phi$ increases and the same inequality $\phi^p(s) -\om \phi(s)=\phi(\phi^{p-1}(s)-\om)>0$ will continue to hold, enforcing $\phi'>0$ through \eqref{810}. Thus, starting from zero, $\rho  \to \phi(\rho)$ will be globally increasing function, which contradicts the localization of $\phi$, i.e. $\int_0^\infty \phi^2(\rho) \rho^{d-1} d\rho<\infty$. 
	Thus, $\phi'<0$ (and consequently $\phi$ decreases) in some neighborhood $(0, \rho_0)$. 
	
	From \eqref{810}, we can of course derive 
	\begin{equation}
		\label{812} 
		-\phi'(\rho) \rho^{d-1+2a} \leq \int_0^\rho \phi^p(s) s^{d-1} ds
	\end{equation}
	Multiplying both sides by $\phi^{q-1}$, where $q>1$ to be discussed, we obtain (Note $-\phi'(\rho)=|\phi'(\rho)|$ as $\phi$ is radially decreasing), 
	$$
	|\phi'(\rho)| \phi^{q-1}(\rho) \rho^{d-1+2a}\leq \phi^{q-1}(\rho) \int_0^\rho \phi^p(s) s^{d-1} ds \leq  \int_0^\rho \phi^{p+q-1}(s) s^{d-1} ds,
	$$
	as $\phi^{q-1}(\rho)<\phi^{q-1}(s)$. {\it This is in principle only valid in the range $\rho\in (0, \rho_0)$}. Thus, if 
	$$
	\int_{\rd} \phi^{p+q-1}(x) dx=\int_0^\infty \phi^{p+q-1}(s) s^{d-1} ds<\infty, 
	$$
	then clearly, 
	\begin{equation}
		\label{820} 
		\lim_{\rho\to 0+}  |\phi'(\rho)| \phi^{q-1}(\rho) \rho^{d-1+2a}=0.
	\end{equation}
	This means that one can multiply \eqref{800} by $\phi^{q-1}$ and integrate by parts in $(0, \rho)$ (where the boundary term at $0$ disappears, due to \eqref{820}), {\it and as long as $\int_{\rd}   \phi^{p+q-1}<\infty$},   we obtain 
	$$
	-\phi'(\rho) \phi^{q-1}(\rho) \rho^{d-1+2a} + \f{4(q-1)}{q^2} 
	\int_0^\rho s^{d-1+2a}  (\p_s(\phi^{\f{q}{2}}(s)))^2 ds\leq \int_0^\rho \phi^{p+q-1}(s) s^{d-1} ds. 
	$$
	Since $\phi'(\rho)<0$,  it follows that 
	\begin{equation}
		\label{830} 
		\int_{|x|<\rho_0} |x|^{2a} |\nabla \phi^{\f{q}{2}}|^2 dx\leq C_q \int_{\rd}   \phi^{p+q-1} 
	\end{equation}
	Specializing to a fixed dyadic annulus $|x|\sim 2^{-k}$, we obtain by H\"older's, with $r: 1<r<2$ to be determined,  
	\begin{eqnarray*}
		\left(	\int_{|x|\sim 2^{-k}} |\nabla \phi^{\f{q}{2}}|^r d\right)^{\f{1}{r}} &\leq &  C \left(	\int_{|x|\sim 2^{-k}} |\nabla \phi^{\f{q}{2}}|^2 d\right)^{\f{1}{2}} 2^{-k d\left(\f{1}{r}-\f{1}{2}\right)}\leq \\
		& \leq & C \left(	\int_{|x|\sim 2^{-k}} ||x|^{2a} \nabla \phi^{\f{q}{2}}|^2 d\right)^{\f{1}{2}}  2^{k d\left(a-d(\f{1}{r}-\f{1}{2})\right)}.
	\end{eqnarray*}
	So, select $r: a-d(\f{1}{r}-\f{1}{2})<0$ or $r<\f{2d}{d+2a}$. With this choice, and taking \eqref{830} into account,  it is clear that 
	$$
	\| \nabla \phi^{\f{q}{2}}\|_{L^r(|x|<\rho_0)}\leq C \|\phi\|_{L^{p+q-1}}^{\f{p+q-1}{2}}.
	$$
	Recall that the bound on $\| \nabla \phi^{\f{q}{2}}\|_{L^r(|x|\geq \rho_0)}$ is already available from Proposition \ref{app:pr20}. 
	Sobolev embedding  yields 
	$$
	\| \phi\|_{L^{\f{qrd}{2(d-r)}}}^{\f{q}{2}}= 	\| \phi^{\f{q}{2}}\|_{L^{\f{rd}{d-r}}}\leq  \| \nabla \phi^{\f{q}{2}}\|_{L^r}.
	$$
	We apply this to $q_0:=p^*-p+1=:p_0-p+1$, and thus we step on the  bound $\|\phi\|_{L^{p+q_0-1}}=\|\phi\|_{L^{p^*}}$, and obtain a new improved bound for  $\|\phi\|_{L^s}: s<p_1:=\f{d}{d+2a-2} (p^*-p+1)$.  Starting with the bound, $\|\phi\|_{L^s}: s<p_1:=\f{d}{d+2a-2} (p^*-p+1)$, we can of course iterate  this further  to an improved  {\it a priori} bound procedure, of the form 
	\begin{eqnarray}
		\label{840} 
		\|\phi\|_{L^s}<\infty, s<p_n \Longrightarrow \|\phi\|_{L^s}<\infty, s<p_{n+1}\\
		\label{845}
		p_{n+1}=\f{d}{d+2a-2} (p_n-p+1), p_0=p^*. 
	\end{eqnarray}
	The relation \eqref{845} can be solved exlicitly to 
	$$
	p_n=\f{d}{2-2a}(p-1)+\left(\f{d}{d+2a-2}\right)^n\left(p^*-\f{d}{2-2a}(p-1)\right)
	$$
	Note that since 
	$
	\f{d}{2-2a}(p-1)\leq \f{d}{2-2a}\f{4(1-a)}{d-2(1-a)}=p^*, 
	$
	we have that $\lim_n p_n=+\infty$. This implies {\it a priori} bounds for $\|\phi\|_{L^s}, s<\infty$. 
	
	We now aim for a bound on $\|\phi\|_{L^\infty}$, given that $\|\phi\|_{L^r}\leq C_r, r<\infty$. Going back to \eqref{810}, we derive for all $\rho>0$, and for all $q>1$, 
	$$
	|\phi'(\rho)| \rho^{d-1+2a}\leq \int_0^\rho \phi^p(s) s^{d-1} ds \leq 
	\left(\int_0^\rho \phi^{p q}(s) s^{d-1} ds  \right)^{\f{1}{q}} 
	\left(\int_0^\rho s^{d-1} ds  \right)^{\f{1}{q'}} \leq \|\phi\|_{L^{pq}}^p \rho^{\f{d}{q'}},
	$$
	by H\"older's. Selecting $q$ large enough, namely $q>\f{d}{2(1-a)}$. As a consequence,  we have that $\f{d}{q'}-(d-1+2a)>-1$, which implies that 
	$|\phi'(\rho)| \leq C \rho^{\f{d}{q'}-(d-1+2a)}$ is integrable on $(0,1)$. Thus, 
	$\phi(\rho)=\phi(1)-\int_\rho^1 \phi'(\rho)d\rho$ has a limit as $\rho\to 0+$, which implies the bound on $\|\phi\|_{L^\infty}$. 
	
	We finally address the behavior of $\phi'(\rho), \phi''(\rho)$ close to zero. From \eqref{810} and the fact that $\phi^p(0)-\om \phi(0)>0$ (recall otherwise $\phi'$ starts to increase at $\rho=0$ and hence cannot be localized), we obtain the formula 
	\begin{eqnarray*}
		\phi'(\rho) &=& -\rho^{1-2a-d} \int_0^\rho (\phi^p(s)-\om \phi(s)) s^{d-1} ds = 
		-\rho^{1-2a-d} (\phi^p(0)-\om \phi(0)) \f{\rho^d}{d}+o(\rho^{1-2a}) \\
		&=& -\f{\phi^p(0)-\om \phi(0)}{d} \rho^{1-2a}+o(\rho^{1-2a}). 
	\end{eqnarray*}
	Regarding $\phi''$, we take a derivative in \eqref{810}. This, together with the formula for $\phi'(\rho)$  leads to 
	\begin{eqnarray*}
		& & -\phi''(\rho)\rho^{d-1+2a}  + (d-1+2 a) \rho^{d-1} \f{\phi^p(0)-\om \phi(0)}{d}+o(\rho^{d-1})=(\phi^p(\rho)-\om \phi(\rho)) \rho^{d-1} =\\
		&=& (\phi^p(0)-\om \phi(0))\rho^{d-1}+o(\rho^{d-1}).
	\end{eqnarray*}
	Solving for $\phi''(\rho)$ yields 
	$$
	\phi''(\rho) = \f{2a-1}{d} (\phi^p(0)-\om \phi(0)) \rho^{-2a} + o(\rho^{-2a}).
	$$
	
\end{proof}


\end{document}